\documentclass[12pt]{article}
\usepackage{amsmath}
\usepackage{textcomp}
\usepackage{amsbsy}
\usepackage{latexsym}
\usepackage[mathscr]{eucal}
\usepackage{amsfonts}
\usepackage{amsthm}
\usepackage{amssymb}
\usepackage{algorithmic}
\usepackage{algorithm}
\usepackage{subfigure}

\usepackage[usenames,dvipsnames]{xcolor}
\usepackage{tikz,ifthen,fp}%,fullpage}
\usepackage{pgfplots}
\usetikzlibrary{backgrounds,arrows}

\usepackage{hyperref}
\usepackage[makeroom]{cancel}

\usepackage{fullpage}

\newtheorem{lemma}{Lemma}[section]
\newtheorem{proposition}[lemma]{Proposition}
\newtheorem{corollary}[lemma]{Corollary}

\newtheorem{remark}[lemma]{Remark}

\begin{document}

\bibliographystyle{plain}

\title{Projection based model order reduction methods for the estimation of vector-valued variables of interest\thanks{This work was supported by the French National Research Agency (Grant ANR CHORUS MONU-0005)}
}

\author{ Olivier ZAHM\footnotemark[4], Marie BILLAUD-FRIESS\footnotemark[2], Anthony NOUY\footnotemark[2]\ \footnotemark[3] }

\renewcommand{\thefootnote}{\fnsymbol{footnote}}
\footnotetext[2]{Ecole Centrale de Nantes, GeM, UMR CNRS 6183, France.}
\footnotetext[3]{Corresponding author (anthony.nouy@ec-nantes.fr).}

\renewcommand{\thefootnote}{\arabic{footnote}}

\maketitle
\date{}

\renewcommand{\thefootnote}{\fnsymbol{footnote}}
\footnotetext[2]{Ecole Centrale de Nantes, GeM, UMR CNRS 6183, France.}
\footnotetext[3]{Corresponding author anthony.nouy@ec-nantes.fr.}
\footnotetext[4]{Department of Aeronautics \& Astronautics, Massachusetts Institute of Technology,
Cambridge, MA 02139, USA.}

\renewcommand{\thefootnote}{\arabic{footnote}}
\newcommand{\MV}{\mathrm{\bf V}}
\newcommand{\MW}{\mathrm{\bf W}}
\newcommand{\MT}{\mathrm{\bf T}}

\begin{abstract}
  We propose and compare goal-oriented projection based model order reduction methods for the estimation of vector-valued functionals of the solution of parameter-dependent equations. The first projection method is a generalization of the classical primal-dual method to the case of vector-valued variables of interest. We highlight the role played by three reduced spaces: the approximation space and the test space associated to the primal variable, and the approximation space associated to the dual variable. Then we propose a Petrov-Galerkin projection method based on a saddle point problem involving an approximation space for the primal variable and an approximation space for an auxiliary variable. A goal-oriented choice of the latter space, defined as the sum of two spaces, allows us to improve the approximation of the variable of interest compared to a primal-dual method using the same reduced spaces. Then, for both approaches, we derive computable error estimates for the approximations of the variable of interest and we propose greedy algorithms for the goal-oriented construction of reduced spaces. The performance of the algorithms are illustrated on numerical examples and compared to standard (non goal-oriented) algorithms.
\end{abstract}

%%%%%%%%%%%%%%%%%%%%%%%%%%%%%%%%%%%%%%%%%%%%%%%%%%%%%%%%%%%%%%%%%%%%%%%%%%%%%%%%%%%%%%%%%%%
%%%%%%%%%%%%%%%%%%%%%%%%%%%%  INTRO  %%%%%%%%%%%%%%%%%%%%%%%%%%%%%%%%%%%%%%%%%%%%%%%%%%%%%%
%%%%%%%%%%%%%%%%%%%%%%%%%%%%%%%%%%%%%%%%%%%%%%%%%%%%%%%%%%%%%%%%%%%%%%%%%%%%%%%%%%%%%%%%%%%
 \section{Introduction}

This paper is concerned with the numerical solution of  linear equations of the form 
\begin{equation} \label{eq:paramdep}
A(\xi) u(\xi) = b(\xi),
\end{equation}
where the operator $A(\xi)$ and right-hand side $b(\xi)$ depend on a parameter $\xi$  which takes values in some parameter set $\Xi$.  Such equations arise in many contexts such as uncertainty quantification, optimization or control, where the solution of  \eqref{eq:paramdep} have to be evaluated with many instances of the parameters (multi-query context).
For large systems of equations (e.g. arising from a fine discretization of a parameter-dependent partial differential equation), solving \eqref{eq:paramdep} for one instance of the parameter can be very expensive,  which leads to intractable computations in a multi-query context. Model order reduction methods aim at constructing an approximation of the solution map $u: \Xi \to V$ whose evaluation for a certain value of $\xi$ is cheaper than solving \eqref{eq:paramdep}. Standard approaches rely on Galerkin-type projections of $u(\xi)$ on a low-dimensional subspace $V_r$ of the solution space $V$, a so-called \emph{reduced space}. The reduced space can be generated from evaluations (snapshots) of the solution $u(\xi)$ at some selected (or randomly chosen) values of the parameter $\xi$, see \cite{Berkooz93,Kahlbacher07,Prudhomme2002,Rozza2008}. The Proper Orthogonal Decomposition method  aims at constructing an optimal subspace for the approximation of the set of solutions $\mathcal{M}=\{u(\xi):\xi\in\Xi\}$ in a mean-square sense (see \cite{Kahlbacher07}). Reduced Basis (RB) methods (see \cite{haasdonk2014reduced} for a survey) aim at controlling the approximation uniformly over the parameter set. In this context, reduced spaces are usually constructed using greedy algorithms.

In many applications one is not interested in the solution $u(\xi)$ itself, but only in a variable of interest $s(\xi)$ which is a functional of $u(\xi)$. Here we assume that $s(\xi)$ depends linearly on $u(\xi)$. Efficient goal-oriented methods have been proposed for the estimation of a scalar-valued variable of interest $s(\xi)$. A standard method consists in computing an approximation of the solution of the so-called dual problem associated to \eqref{eq:paramdep} which is used to correct the estimation of $s(\xi)$. We refer to \cite{Pierce2000} for a general survey on primal-dual methods and to \cite{Chen10,Grepl2005,haasdonk2014reduced,Prudhomme2002} for the application in the context of RB methods.
~\\

{
In this paper, we propose projection based model order reduction methods for the estimation of a variable of interest $s(\xi)$ taking values in a vector space of finite or infinite dimension. 
We consider the case where
\begin{equation*}
 s(\xi) = L(\xi) u(\xi),
\end{equation*}
with $L(\xi)$ a parameter-dependent linear operator. 
For example, for boundary value problems, $L(\xi)$ can be defined as the trace operator providing the restriction of the solution to the boundary of the domain. In this case the variable of interest belongs to an infinite dimensional space or, after discretization, to a finite but possibly high dimensional space.
The standard approach, which consists in treating $s(\xi)$ as a collection of scalar-valued variables of interest and in building one reduced dual space for each of them, has a complexity which grows proportionally to the dimension of $s(\xi)$. Our approach circumvents this issue by constructing a single reduced dual space, thus allowing to handle variables of interest with high and potentially infinite dimension.
A similar approach can be found for parametric dynamical systems, see the monograph \cite{benner2005dimension} for a general introduction. In this framework, projection-based model order reduction methods are used for the approximation of $s(\xi)$ which is an output of the dynamical system. Petrov-Galerkin methods have been proposed with different ways of constructing the reduced basis for the test and trial space, such as the balanced truncation methods, (balanced) Proper Orthogonal Decomposition method, moment matching methods, etc. We refer to \cite{BGW2015} for a recent review on these methods. In the present paper, we aim at exploring other possibilities than the Petrov-Galerkin projection.
}
~\\

In a first part, we introduce and analyze different methods for computing projections of the solution and approximations of the variable of interest. 
We first present a non goal-oriented Petrov-Galerkin approach to compute an approximation of $u(\xi)$ from which an estimation of $s(\xi)$ is deduced.
Then, we introduce a generalization of the standard primal-dual method to the case of a vector-valued variable of interest, which relies on the approximation of the primal variable $u(\xi)$ and of the solution $Q(\xi)$ of the dual problem
\begin{equation*}
A(\xi)^* Q(\xi) = L(\xi)^*,
\end{equation*}
where $A(\xi)^*$ and $L(\xi)^*$ are the adjoints of operators $A(\xi)$ and $L(\xi)$ respectively. We show that  the error on the variable of interest depends on three reduced spaces:  the  approximation space $V_r$ for the primal variable $u(\xi)$,  the test space $W_r$ which is used for the Petrov-Galerkin projection of $u(\xi)$, and  an approximation space $W_k^Q$ for the dual variable $Q(\xi)$ which is projected on the space of $W_k^Q$-valued linear operators. 
{Finally}, we present a  Petrov-Galerkin method where the projection is obtained by solving a saddle point problem which involves an approximation space $V_r$ for $u(\xi)$ and an approximation space $T_p$ for an auxiliary variable. We show that if $T_p$ is {defined by} $T_p = W_r + W_k^Q$, then error bounds for both the projection of the primal variable on $V_r$ and the approximation of the variable of interest can be improved compared to error bounds of a primal-dual approach using the same spaces $V_r$, $W_r$ and $W_k^Q$.  The proposed  approach is a goal-oriented extension of the method proposed in \cite{Dahmen2013}. 
 
In a second part, we derive (for both approaches) computable error estimates for the approximation of the variable of interest. Then, we propose greedy algorithms based on these error estimates for the construction of the reduced spaces $V_r$ and $W^Q_k$. We discuss different choices for the reduced space $W_r$. In particular, we introduce a parameter-dependent space depending on a preconditioner obtained by means of an interpolation of the inverse of the operator $A(\xi)$ proposed in \cite{zahm2015interpolation}. 
\\\par
This paper is organized as follows. In Section \ref{sec:Analysis}, we introduce and analyze the different projection methods for the estimation of vector-valued variables of interest for general linear equations of the form \eqref{eq:paramdep} formulated in a Hilbert setting. Then, in Section \ref{sec:Application2MOR}, we derive error estimates for the approximation of the variable of interest and we propose practical greedy algorithms for the construction of reduced spaces. Finally, in Section \ref{sec:Numerical}, numerical experiments illustrate the properties of the projection methods and of the greedy algorithms. In particular, we provide a simplified complexity  analysis for the so-called \emph{offline} phase (\emph{i.e.} the construction of the reduced spaces) and for the \emph{online} phase (\emph{i.e.} the evaluation of $s(\xi)$ for a particular instance of $\xi$).

%%%%%%%%%%%%%%%%%%%%%%%%%%%%%%%%%%%%%%%%%%%%%%%%%%%%%%%%%%%%%%%%%%%%%%%%%%%%%%%%%%%%%%%%%%%
%%%%%%%%%%%%%%%%%%%%%%%%%%%%  Analysis %%%%%%%%%%%%%%%%%%%%%%%%%%%%%%%%%%%%%%%%%%%%%%%%%%%%
%%%%%%%%%%%%%%%%%%%%%%%%%%%%%%%%%%%%%%%%%%%%%%%%%%%%%%%%%%%%%%%%%%%%%%%%%%%%%%%%%%%%%%%%%%%
\section{Projection methods for the estimation of a variable of interest}\label{sec:Analysis}

Let $V$, $W$ and $Z$ be three Hilbert spaces. For a Hilbert space $H$ equipped with a norm $\|\cdot\|_H$, we denote by $H'$ the {topological dual space} of $H$. We consider the linear equation
\begin{equation}\label{eq:gen}
Au=b
\end{equation}
with $A \in \mathcal{ L}(V, W')$ and $b\in W'$, and a variable of interest 
\begin{equation*}
 s=Lu,
\end{equation*}
where $L\in \mathcal{L}(V,Z)$.
We assume that $A$ is a norm-isomorphism\footnote{$A$ is a norm-isomorphism if it is a continuous and weakly coercive operator satisfying the assumptions of the Ne\v cas' theorem \cite[Chapter 2]{Ern}. }
such that for all $u\in V$,
\begin{equation*}
\alpha \|u\|_V \le \|Au\|_{W'} \le \beta \|u\|_V, 
\end{equation*}
where   
\begin{subequations}\label{infsup}
\begin{align}
&\displaystyle \inf_{0\neq v \in V} \sup_{0\neq w\in W} \frac{\langle Av,w \rangle}{\Vert v \Vert_V \Vert w \Vert_W} := \alpha >0\label{infsup_a}\\
& \sup_{0\neq v \in V} \sup_{0\neq w\in W} \frac{\langle Av,w \rangle}{\Vert v \Vert_V \Vert w \Vert_W} := \beta <\infty,\label{infsup_b}
\end{align}
\end{subequations}
which ensures the well-posedness of \eqref{eq:gen}. In this section, we present different methods for constructing an approximation $\widetilde s$ of $s$. {First, in Section \ref{sec:PG}, we present a standard approach which consists in estimating the variable of interest from a Petrov-Galerkin projection of $u$. In Section \ref{sec:PD}, we present an extension of the primal-dual approach to the case of vector-valued variables of interest, where the  variable of interest is estimated from a standard Petrov-Galerkin projection of the primal variable and a projection of the solution of a dual problem. Finally, in Section \ref{sec:SaddlePoint}, we introduce a goal-oriented projection method based on a saddle-point formulation.}

Before going further, let us introduce some additional notations. For a Hilbert space $H$, we denote by $R_H \in \mathcal{L}(H,H')$ the Riesz map such that $\|v\|_H^2= {\langle R_H v,v \rangle}$, where $\langle\cdot , \cdot \rangle$ denotes the duality pairing. The dual norm $\|\cdot\|_{H'}$ on $H'$ is such that $R_{H'}=R_H^{-1}$. Then $ \| v\|_H=\| R_Hv \|_{H'} $ and $|\langle v,w \rangle|\leq \| v \|_H\| w \|_{H'}$ hold for any $v\in H$ and $w\in H'$. For any operator $C \in \mathcal{L}(H_1,H_2')$, with $H_1$ and $H_2$ two Hilbert spaces, $C^*\in \mathcal{L}(H_2,H_1')$ denotes the adjoint of $C$, such that $\langle Cv_1,v_2 \rangle = \langle v_1, C^* v_2 \rangle$ for any $v_1\in H_1$ and $v_2\in H_2$.

\subsection{Petrov-Galerkin projection}\label{sec:PG}

Suppose that we are given a subspace $V_r \subset V$ of finite dimension $r$ in which we seek an approximation of $u$. The orthogonal projection $u_r^\perp$ of $u$ on $V_r$, given by $\|u -u_r^\perp \|_V=\min_{v\in V_r}\| u-v \|_V$, is characterized by 
\begin{equation}\label{eq:def_orth_proj}
 \langle u-u_r^\perp,R_V v\rangle =0,\quad \forall v\in V_r.
\end{equation}
In practice, an approximation 
$u_r\in V_r$ can be defined as a Petrov-Galerkin projection of $u$ characterized by
\begin{equation}\label{eq:def_PG}
 \langle Au_r - b,y\rangle =0, \quad \forall y\in W_r,
\end{equation}
where $W_r\subset W$ is a test space of dimension $r$. Under the assumption that
\begin{align}
\alpha_{V_r,W_r}  = { \inf_{0\neq v \in V_r}  \sup_{0\neq y\in W_r} } \frac{\langle Av,y\rangle}{\Vert v \Vert_V \Vert y \Vert_W}>0,\label{eq:infsup_discrete_PG}
\end{align}
the next proposition provides a quasi-optimality result for $u_r$ and gives an error bound for the approximation of the variable of interest.
{In what follows, notation $\min$ (resp. $\max$) is used in place of $\inf$ (resp. $\sup$) when the minimum (resp. the maximum) is reached.}

\begin{proposition} \label{prop:1}
Under assumption \eqref{eq:infsup_discrete_PG},
the solution $u_r$ of equation \eqref{eq:def_PG} satisfies
 \begin{equation}
  \| u - u_r  \|_V \leq \frac{1}{\sqrt{1-(\delta_{V_r,W_r})^2}} \min_{v\in V_r}\| u-v \|_{V}. \label{eq:control_PG}
 \end{equation}
 where 
  \begin{equation}
  \delta_{V_r,W_r} = \max_{0\neq v\in V_r}\min_{ y\in W_r}\frac{\| v -R_V ^{-1}A^*y \|_V}{\| v \|_V}<1. \label{eq:def_delta_r} 
 \end{equation}
 Furthermore,
 \begin{align}
  \| s - Lu_r \|_Z \leq \frac{\delta_{W_r}^L}{\sqrt{1-(\delta_{V_r,W_r})^2}} \min_{v\in V_r}\| u-v \|_{V},\label{eq:control_VI_PG_2}
 \end{align}
 with
 \begin{equation}
 \delta_{W_r}^L = \sup_{0\neq z'\in Z'} \min_{y\in W_r} \frac{\| L^* z' -A^{*}y \|_{V'}}{\| z' \|_{Z'}}. \label{eq:def_delta_r_L}
 \end{equation}

\end{proposition} 

\begin{proof}
 With $u_r^\perp$  the orthogonal projection of $u$ on $V_r$, for any $v\in V_r \setminus \{0\}$ and $ y\in W_r$, we have
 \begin{align*}
  \langle u_r^\perp-u_r,R_V v \rangle &\overset{\eqref{eq:def_orth_proj}}{=} \langle u-u_r,R_V v \rangle = \langle b-Au_r,A^{-*}R_V v \rangle \\
  &\overset{\eqref{eq:def_PG}}{=} \langle b-Au_r,A^{-*}R_V v -y \rangle = \langle u-u_r,R_V v -A^{*}y \rangle \\
  &\leq \| u-u_r \|_{V}  \| R_V v -A^{*}y  \|_{V'}.
 \end{align*}
 Taking the minimum over $  y\in W_r$, dividing by $\| v \|_V$ and taking the maximum over $v\in V_r \setminus\{0\}$, we obtain $ \| u_r^\perp - u_r \|_V \leq \delta_{V_r,W_r} \| u-u_r \|_{V} $, where $\delta_{V_r,W_r}$ is defined by \eqref{eq:def_delta_r}. Thanks to the orthogonality condition \eqref{eq:def_orth_proj} we have
  $ \| u-u_r \|_V^2 = \| u-u_r^\perp \|_V^2 + \| u_r^\perp-u_r\|_V^2$, from which we deduce that 
  $(1-\delta_{V_r,W_r}^2)  \| u-u_r \|_V^2 \le  \| u-u_r^\perp \|_V^2 $. To prove  \eqref{eq:control_PG}, it remains to prove that $\delta_{V_r,W_r}<1$. Noting that 
  \begin{align*}
  \min_{ y\in W_r}\| v -R_V ^{-1}A^*y \|_V^2 &=  \min_{ 0\neq y\in W_r} \min_{\lambda\in\mathbb{R}}\| v -\lambda R_V ^{-1}A^*y \|_V^2 =   \min_{ 0\neq y\in W_r} \| v \|_V^2 -  \frac{\langle v,A^*y\rangle^2}{\| A^*y \|_{V'}^2},
  \end{align*}
we obtain 
\begin{align}
\delta_{V_r,W_r}^2  &= 1 - \min_{0\neq v\in V_r}\max_{ 0\neq y\in W_r}\frac{\langle Av,y\rangle^2}{\| v \|_V^2 \| A^*y \|_{V'}^2}.\label{deltar_relation1}
\end{align}
{Let introduce $\beta_{W_r} = \sup_{0\neq y_r \in W_r} \| A^* y \|_{V'} / \Vert y \Vert_W$ which, from assumption \eqref{infsup_b}, satifies $\beta_{W_r} \leq \beta $. Then using assumption \eqref{eq:infsup_discrete_PG} we obtain }
\begin{align}
\delta_{V_r,W_r}^2 \le 1 - \frac{\alpha_{V_r,W_r}^2}{\beta_{W_r}^2} \le 1 - \frac{\alpha_{V_r,W_r}^2}{\beta^2} <1. \label{ineq_delta_alpha}
\end{align}
   Furthermore for any $z'\in Z' \setminus\{0\}$ and $ y\in W_r$, we have
  \begin{align*}
    \langle s-Lu_r, z' \rangle &= \langle b-Au_r, A^{-*}L^* z' \rangle \overset{\eqref{eq:def_PG}}{=} \langle b-Au_r, A^{-*}L^* z' -y \rangle \\
    &\leq \| u-u_r \|_{V}  \| L^* z' -A^{*}y  \|_{V'}.
  \end{align*}
  Taking the infimum over $ y\in W_r$, dividing by $\| z' \|_{Z'}$ and taking the supremum over $z' \in Z' \setminus\{0\}$, we obtain \eqref{eq:control_VI_PG_2} thanks to \eqref{eq:control_PG}.
\end{proof}

The error bound \eqref{eq:control_VI_PG_2} for the approximation of the variable of interest $s$ is the product of three terms: 
\begin{itemize}
 \item[(a)] $\inf_{v\in V_r}\| u-v \|_V$, which suggests that the approximation space $V_r$ should be defined such that $u$ can be well approximated in $V_r$, 
 \item[(b)] $(1-(\delta_{V_r,W_r})^2)^{-1/2}$, which suggests that the test space $W_r$ should be chosen such that any element of $V_r$ can be well approximated by an element of $R_V^{-1}A^* W_r$, and 
 \item[(c)] $\delta^L_{W_r}$, which suggests that any element of $\text{range}(L^*)$ should be well approximated by an element of $A^*W_r$.
\end{itemize}
As already noticed in  \cite[Section 11.1]{Rozza2008}, $W_r$ plays a double role: a test space for the definition of $u_r$ (point (b)) and an approximation space for the range of $A^{-*}L^*$ (point (c)). 

{
\begin{remark}
 The proposed Petrov-Galerkin projection method coincides with the interpolatory projection method used in the context of parametric dynamical systems (see \cite{baur2011,BGW2015}). Our analysis provides quasi-optimality results on $s(\xi)$ for any parameter value $\xi$. Also, the condition $\delta_{V_r,W_r}(\xi) > 0$ ensures the invertibility of the reduced operator $A_r(\xi) : V_r \to W_r'$ defined by $\langle A_r(\xi) v , y \rangle = \langle A(\xi) u_r , y \rangle$ for all $v\in V_r$ and $y\in W_r$. In \cite{baur2011}, the invertibility of $A_r(\xi)$ is not discussed in the time-independent case. 
\end{remark}
}

\begin{remark}[Comparison with the C\'ea's Lemma]
~Under assumption \eqref{eq:infsup_discrete_PG}, ~the classical C\'ea's lemma states that
\begin{align}
  \| u - u_r  \|_V \leq (1+\frac{\beta}{\alpha_{V_r,W_r}})~\min_{v\in V_r}\| u-v \|_{V}.\label{cea}
\end{align}
{
It has been shown in \cite{Xu} that this can be improved to
\begin{align}\label{eq:improved_cea}
  \| u - u_r  \|_V \leq \frac{\beta}{\alpha_{V_r,W_r}}\min_{v\in V_r}\| u-v \|_{V}.
\end{align}
Noting that Equation \eqref{ineq_delta_alpha} yields
\begin{equation*}
 \frac{1}{\sqrt{1-(\delta_{V_r,W_r})^2}} \le  \frac{\beta_{W_r}}{\alpha_{V_r,W_r}} \le  \frac{\beta}{\alpha_{V_r,W_r}},
\end{equation*}
we observe that \eqref{eq:control_PG} provides a sharper bound than in \eqref{eq:improved_cea}, where the constants differ by a factor $\beta_{W_r}/\beta$.
}

\end{remark}

\begin{remark}[Symmetric coercive case and compliant case]\label{rmk:SPD_1}
 We suppose that $A$ is a symmetric coercive operator, with $V=W$ and $\Vert \cdot\Vert_V = \Vert \cdot\Vert_W$ the norm induced by the operator $A$ such that $R_V=A$. Then $\delta_{V_r,W_r}$ defined by \eqref{eq:def_delta_r} admits the following simple expression
 \begin{equation*}
  \delta_{V_r,W_r} = \min_{0\neq v\in V_r}\max_{ y\in W_r}\frac{\| v - y \|_V}{\| v \|_V}.
 \end{equation*}
 If the test space $W_r$ is defined by $W_r=V_r$, we obtain $\delta_{V_r,W_r}=0$, and from \eqref{eq:control_PG}, we obtain $u_r=u_r^\perp$. In other words, the standard Galerkin projection coincides with the orthogonal projection.
 
 In the case where the variable of interest $s$ is scalar-valued, we have $Z=\mathbb{R}$ and $\mathcal{L}(V,Z)=V'$.
 The so-called \textit{compliant case} corresponds to $Lv=\langle b,v\rangle$ for any $v \in V$. Then, by definition \eqref{eq:def_delta_r_L}, we have 
 \begin{equation*}
  \delta_{W_r}^L=\min_{v\in V_r}\| b-Av \|_{V'} = \min_{v\in V_r}\| u-v \|_V =\| u-u_r \|_V,
 \end{equation*}
 and thanks to \eqref{eq:control_VI_PG_2}, we recover the so-called ``squared effect''
 \begin{equation*}
   | s - Lu_r | = \| s - Lu_r \|_Z \leq \|u-u_r \|^2_V.
 \end{equation*}
\end{remark}

\subsection{Primal-dual approach}\label{sec:PD}
~We now extend the classical primal-dual approach \cite{Pierce2000} for the estimation of a vector-valued variable of interest.\\

Let us introduce the dual variable $Q\in\mathcal{L}(Z',W)$ defined by $A^*Q = L^*$. The relation
\begin{equation*}
 s = Lu = Q^*Au = Q^* b
\end{equation*}
shows that the variable of interest can be exactly determined if either the primal variable $u$ or the dual variable $Q$ is known. \\

Now,  for given  approximations $\widetilde u$ of $u$ and $\widetilde Q$ of $Q$, we define the approximation  $\tilde s$ of $s$ by 
 \begin{equation}\label{eq:def_st_uQ}
  \widetilde s = L\widetilde u + \widetilde Q^*(b-A\widetilde u),
 \end{equation}
where $L\widetilde u$ is the standard estimation of the variable of interest and where $\widetilde Q^*(b-A\widetilde u)$ is a correction using the approximation of the dual variable. The following proposition provides an error bound on the variable of interest, which is a generalization of the classical error bound for scalar-valued variables of interest (see \cite{Pierce2000}) to vector-valued variables of interest.\\

\begin{proposition}\label{prop:ClassicalBounds}
The approximation $\tilde s$ of $s$ defined by \eqref{eq:def_st_uQ} satisfies
 \begin{equation}\label{eq:bound_s_uQ}
  \| s-\widetilde s\|_Z \leq \| u-\widetilde u\|_V \| L^* - A^* \widetilde Q \|_{Z'\rightarrow V'},
 \end{equation}
 where
 \begin{equation}
  \| L^* - A^* \widetilde Q \|_{Z'\rightarrow V'} = \sup_{0\neq z'\in Z'} \frac{\| (L^* - A^* \widetilde Q)z' \|_{V'}}{\| z' \|_{Z'}}.
 \end{equation}

\end{proposition}

\begin{proof}
 For any $z'\in Z'$, we have
 \begin{align*}
  \langle s-\widetilde s,z' \rangle &= \langle Lu-L\widetilde u - \widetilde Q^*(b-A\widetilde u), z' \rangle 
  = \langle (L-\widetilde Q^*A)(u-\widetilde u), z' \rangle \\
  &= \langle u-\widetilde u,(L^*- A^*\widetilde Q) z' \rangle 
  \leq \| u-\widetilde u\|_V \| (L^*- A^*\widetilde Q) z' \|_{V'}.
 \end{align*}
 Dividing by $\| z' \|_{Z'}$ and taking the supremum over $z'\in Z'\setminus \{0\}$, we obtain \eqref{eq:bound_s_uQ}.
\end{proof}

~\\
In practice, the approximation $\widetilde u $ can be defined as the Petrov-Galerkin projection $u_r$ of $u$ on a given approximation space $V_r$ with a given test space $W_r$, see equation \eqref{eq:def_PG}. For the approximation $\widetilde Q$ of $Q \in \mathcal{L}(Z',W)$, the bound \eqref{eq:bound_s_uQ} suggests that $\| L^* - A^* \widetilde Q \|_{Z'\rightarrow V'}$ should be small. We then propose to choose $\widetilde Q$ as a solution of
\begin{equation}
\min_{\widetilde Q\in \mathcal{L}(Z',W_k^Q)} \| L^* - A^* \widetilde Q \|_{Z'\rightarrow V'} \label{eq:optimal_Qk},
\end{equation}
where $W_k^Q\subset W$ is a given approximation space (different from $W_r$). The next proposition shows how to construct a solution of \eqref{eq:optimal_Qk}.\\

\begin{proposition} \label{def:qk}
 The operator $Q_k: Z'\rightarrow W_k^Q$ defined for $z'\in Z'$ by
 \begin{equation}
Q_k z'   =\arg\min_{y_k\in W_k^Q} \| L^*z' - A^* y_k \|_{V'} \label{eq:def_Qk_1}
 \end{equation}
 is linear and is a solution of \eqref{eq:optimal_Qk}. Moreover $Q_kz' \in W_k^Q$ is characterized by
 \begin{equation}\label{eq:def_Qk_2}
  \langle  L^*z'- A^* Q_k z', R_V^{-1} A^* y_k \rangle =0, \qquad \forall y_k\in W_k^Q.
 \end{equation}
 
\end{proposition}

\begin{proof}
We easily prove that the optimization problem \eqref{eq:def_Qk_1} admits a unique solution which depends linearly and continuously on $z'$, so that $Q_k$ defined by \eqref{eq:def_Qk_1} is a linear operator in $\mathcal{L}(Z',W_k^Q)$. Equation \eqref{eq:def_Qk_2} is the Euler equation associated to the minimization problem \eqref{eq:def_Qk_1}. Furthermore for any $\widetilde Q \in \mathcal{L}(Z',W_k^Q)$ and $z'\in Z'\setminus \{0\}$, we have
\begin{equation*}
 \frac{\| L^*z' - A^* Q_kz' \|_{V'}}{\| z' \|_{Z'}} \overset{\eqref{eq:def_Qk_1}}{\leq} \frac{\| L^*z' - A^* \widetilde Q z' \|_{V'}}{\| z' \|_{Z'}} \leq \| L^*- A^* \widetilde Q \|_{Z'\rightarrow V'}.
\end{equation*}
Taking the supremum over $z'\in Z'\setminus\{0\}$ and then the infimum over $\widetilde Q \in \mathcal{L}(Z',W_k^Q)$, we obtain that $\| L^*- A^*  Q_k \|_{Z'\rightarrow V'} \le \inf_{\widetilde Q\in \mathcal{L}(Z',W_k^Q)}\| L^*- A^* \widetilde Q \|_{Z'\rightarrow V'}$, which means that $Q_k\in \mathcal{L}(Z',W_k^Q)$ is a solution of \eqref{eq:optimal_Qk}.
\end{proof}

~\\
In practice, for computing the approximation of the variable of interest \eqref{eq:def_st_uQ} with $\widetilde Q=Q_k$, we only need to compute $Q_k^*(b-Au_r)$. The following lemma shows how this can be performed without computing the operator $Q_k$.\\

\begin{lemma}\label{lem:compute_QtR}
 Let $Q_k$ be defined by \eqref{eq:def_Qk_1}. Then for $r = b-Au_r \in W'$, 
 \begin{equation}
 Q_k^*r = LR_V^{-1}A^*y_k^\perp, 
 \end{equation}
 where $y_k^\perp\in W_k^Q$ is defined by
 \begin{equation}
  \langle AR_V^{-1}A^* y_k^\perp,y_k\rangle = \langle r,y_k\rangle , \qquad \forall y_k\in W_k^Q.\label{eq:def_ystar}
 \end{equation}
\end{lemma}

\begin{proof}
 For any $z'\in Z'$, since $Q_kz'\in W_k^Q$, we have
 \begin{equation}\label{tmp:9345}
  \langle Q_k z', AR_V^{-1} A^* y_k^\perp \rangle \overset{\eqref{eq:def_ystar}}{=} \langle Q_kz', r \rangle. 
 \end{equation}
 Furthermore, by definition of $Q_k$ we have
 \begin{equation}\label{tmp:3659}
  \langle Q_k z', AR_V^{-1} A^* y_k^\perp \rangle  \overset{\eqref{eq:def_Qk_2}}{=} \langle L^*z' , R_V^{-1}A^* y_k^\perp\rangle.
 \end{equation}
 Combining \eqref{tmp:9345} and \eqref{tmp:3659}, we obtain $ \langle z', Q_k^* r \rangle = \langle z', L R_V^{-1}A^* y_k^\perp \rangle$ for all $z'\in Z'$, which concludes the proof.
\end{proof}

~\\
{We give now a new bound of the error on the variable of interest.}
%{We give now a sharper bound of the error on the variable of interest taking the advantage of the property  \eqref{eq:def_Qk_2}.}\\

\begin{proposition}\label{prop:NewBounds}
  Let $\widetilde u=u_r$ be the Petrov-Galerkin  projection defined by \eqref{eq:def_PG} and let $\widetilde Q=Q_k$ be defined by \eqref{eq:def_Qk_2}. Then the approximation $\widetilde s$ defined by \eqref{eq:def_st_uQ} satisfies
 \begin{align}
  \| s - \widetilde s \|_Z &\leq \delta_{W_k^Q}^L \min_{y_k\in W_k^Q} \| u-u_r - R_V^{-1}A^* y_k \|_{V} , \label{eq:control_VI_PD_1}
 \end{align}
 where
 \begin{align}
  \delta_{W_k^Q}^L &= \sup_{0\neq  z'\in Z'}\min_{y\in W_k^Q} \frac{\| L^*z' -A^* y \|_{V'}}{\| z' \|_{Z'}}.
 \end{align}
 Moreover,
  \begin{align}
  \| s - \widetilde s \|_Z \leq \frac{\delta_{W_k^Q}^L}{\sqrt{1- (\delta_{V_r,W_r})^2}}  \min_{v\in V_r} \| u-v \|_{V}. \label{eq:control_VI_PD_2}
 \end{align}
 
\end{proposition}

\begin{proof}
 For any $z'\in Z'$, and for any $y_k\in W_k^Q$ we have
 \begin{align}
  \langle s-\widetilde s, z' \rangle &\overset{\eqref{eq:def_st_uQ}}{=} \langle u-u_r,(L^*- A^* Q_k) z' \rangle \nonumber\\
  &\overset{\eqref{eq:def_Qk_2}}{=}\langle u-u_r - R_V^{-1}A^* y_k,(L^*- A^* Q_k) z' \rangle \nonumber\\
  &\leq \| u-u_r - R_V^{-1}A^* y_k\|_V \| (L^*- A^*Q_k) z' \|_{V'} \label{eq:tmp6835}.
 \end{align}
From \eqref{eq:def_Qk_1}, we have
 \begin{equation*}
  \| (L^*- A^*Q_k) z' \|_{V'} = \min_{y_k\in W_k^Q} \| L^*z' - A^* y_k \|_{V'}.
 \end{equation*}
 Dividing by $\| z'\|_{Z'}$ and taking the supremum over $z'\in Z'\setminus \{0\}$ in \eqref{eq:tmp6835}, we obtain
 \begin{equation*}
   \| s-\widetilde s \|_Z  \leq  \delta_{W_k^Q}^L \| u-u_r - R_V^{-1}A^* y_k\|_V
 \end{equation*}
 Then, taking the minimum over $y_k\in W_k^Q$, we obtain \eqref{eq:control_VI_PD_1}. Finally, taking $y_k=0$ in \eqref{eq:control_VI_PD_1}, we obtain \eqref{eq:control_VI_PD_2} from \eqref{eq:control_PG}.
\end{proof}

{
\begin{remark}
Observing that $\delta_{W_k^Q}^L  \le  \| L^* - A^* Q_k \|_{V'}$ and $\min_{y_k\in W_k^Q} \| u-u_r - R_V^{-1}A^* y_k \|_{V} \le \| u-u_r\|_V$ , Proposition \ref{prop:NewBounds} provides a sharper bound of the error on the variable of interest by taking advantage of the orthogonality property  \eqref{eq:def_Qk_2}.
\end{remark}
}
\subsection{Projection based on a saddle point problem}\label{sec:SaddlePoint}

In this section we extend the method proposed in \cite{Dahmen2013} for the approximation of (vector-valued) variables of interest. 
The idea is to define the projection of $u$ on the reduced space $V_r$ by means of a saddle point problem. We first define and analyze this saddle point problem. Then we use  the solution of this problem for the estimation of the variable of interest. 
"
Let us equip $W$ with a norm $\|\cdot\|_W$ such that the relation $\| y \|_{W} = \| A^*y\|_{V'}$ holds for any $y\in W$, which is equivalent to the following relation between the Riesz maps $R_W$ and $R_V$:
\begin{equation}\label{eq:idealrieszmap}
R_W = AR_V^{-1}A^*.
\end{equation} 
The orthogonal projection $u_r^\perp$ of $u$ on $V_r$ satisfies
\begin{align*}
 \| u-u_r^\perp \|_V &= \min_{v\in V_r} \| u-v \|_V = \min_{v\in V_r} \sup_{0\neq w\in V} \frac{ \langle u-v,R_V w\rangle }{\| w \|_V} \\
 &= \min_{v\in V_r} \max_{0\neq w\in V} \frac{ \langle Av-b,A^{-*}R_V w\rangle }{\| w \|_V} 
  = \min_{v\in V_r} \max_{0\neq y\in W} \frac{ \langle Av-b,y\rangle }{\| R_V^{-1}A^*y \|_{V}}\\
 &= \min_{v\in V_r} \max_{0\neq y\in W} \frac{ \langle Av-b,y\rangle }{\| y \|_W}.
\end{align*}
Starting from this observation, we introduce a subspace $T_p \subset W$ of dimension $p$ and we define the projection $u_{r,p}$ in $V_r$ as the solution 
of the saddle point problem
\begin{equation}
  \min_{v\in V_r} \max_{\substack{w\in T_p \\ \Vert w\Vert_W=1}} \langle Av-b,w\rangle.\label{eq:def_saddle}
\end{equation}
In the following proposition, we prove the well-posedness of  \eqref{eq:def_saddle} under the condition (discrete inf-sup condition)
\begin{equation} \label{eq:infsup_discrete}
 {\inf_{0 \neq v \in V_r} \sup_{0 \neq  y \in T_p} } \dfrac{\langle A v, y\rangle}{\|v\|_{V} \|y\|_{W}} =:   \alpha_{V_r,T_p}  >0,
 \end{equation}
and we provide a practical characterization of $u_{r,p}$. 

\begin{proposition}\label{prop:saddle_solution}
Under assumption \eqref{eq:infsup_discrete}, there exists a unique solution $(u_{r,p},y_{r,p})$ in $V_r\times T_p$ to 
\begin{subequations}\label{eq:PGD}
 \begin{align}
  \langle R_W y_{r,p} , y\rangle + \langle Au_{r,p}, y\rangle&= \langle b, y\rangle  \quad \forall y\in T_p, \label{eq:PGD_1}\\
  \langle A^* y_{r,p}   ,v\rangle &=0 \quad\quad~~ \forall v\in V_r,\label{eq:PGD_2}
 \end{align}
 \end{subequations}
and $(u_{r,p},\frac{y_{r,p}}{\Vert y_{r,p}\Vert_W})$ is the unique solution of \eqref{eq:def_saddle}.
\end{proposition}
\begin{proof}
Since the Riesz map $R_W$ defined by \eqref{eq:idealrieszmap} is coercive and under the discrete inf-sup condition \eqref{eq:infsup_discrete} on operator $A$, Theorem 2.34 of \cite{Ern} gives that \eqref{eq:PGD} is a well-posed problem whose solution $(u_{r,p},y_{r,p})$ is the unique solution of the saddle-point problem 
\begin{align*}
\min_{v\in V_r} \max_{y\in T_p} -\frac{1}{2} \langle R_W y,y\rangle + \langle b , y \rangle -  \langle Av,y\rangle .
\end{align*}
Denoting $y=\lambda w$ with $\Vert w \Vert_W=1$, this saddle point problem is equivalent to 
\begin{align*}
\min_{v\in V_r} \max_{\substack{w\in T_p \\ \Vert w \Vert_W =1}} \max_{\lambda\in \mathbb{R}} - \frac{1}{2} \lambda^2 + \lambda  \langle b - Av,w\rangle =
\min_{v\in V_r} \max_{\substack{w\in T_p \\ \Vert w\Vert_W =1}}  \frac{1}{2} \langle Av - b,w\rangle^2,
\end{align*}
which coincides with problem \eqref{eq:def_saddle}. 
\end{proof}
~\\

The following proposition provides a quasi-optimality result for the projection $u_{r,p} \in V_r$ of $u$ onto $V_r$.

\begin{proposition}\label{prop:bounds_saddle_point}
Under assumption \eqref{eq:infsup_discrete}, the solution $u_{r,p}$ of \eqref{eq:PGD} is such that 
 \begin{equation}
  \| u-u_{r,p} \|_V \leq \frac{1}{\sqrt{1-\delta_{V_r,T_p}^2}}\min_{v\in V_r} \| u-v \|_V, \label{eq:control_PGD}
 \end{equation}
 where 
  \begin{equation}\label{eq:delta_VrWp}
  \delta_{V_r,T_p} = \max_{0\neq v\in V_r} \min_{y\in T_p} \frac{\|  v -R_V^{-1}A^*y \|_{V}}{\| v \|_V} 
 \end{equation}
 is such that 
 \begin{equation}\label{deltarp_alpharp}
  \delta_{V_r,T_p}^2 = 1 - \alpha_{V_r,T_p}^2 <1.
   \end{equation}

\end{proposition}
 
 \begin{proof} 
Let $(u_{r,p},y_{r,p}) $ be the solution of \eqref{eq:PGD}.
 For any $v\in V_r$ and $y\in T_p$, we have
 \begin{align}
  \langle u_r^\perp-u_{r,p},R_V v \rangle &\overset{\eqref{eq:def_orth_proj}}{=} \langle u-u_{r,p},R_V v \rangle = \langle b-Au_{r,p},A^{-*}R_V v \rangle \nonumber \\
  &\overset{\eqref{eq:PGD_1}}{=} \langle b-Au_{r,p},A^{-*}R_V v -y \rangle + \langle R_Wy_{r,p},y \rangle \nonumber \\
  &\overset{\eqref{eq:PGD_2}}{=} \langle b-Au_{r,p},A^{-*}R_V v -y \rangle - \langle R_W y_{r,p},A^{-*}R_Vv-y \rangle \nonumber \\
  &\overset{\eqref{eq:idealrieszmap}}{=} \langle u-u_{r,p} - R_V^{-1}A^* y_{r,p},R_V v -A^*y \rangle \nonumber \\
  &\leq \| u-u_{r,p} - R_V^{-1}A^* y_{r,p} \|_V \| R_V v -A^*y \|_{V'}. \label{eq:tmp2}
 \end{align}
Equation \eqref{eq:PGD_1} implies that 
$$
y_{r,p} = \arg\min_{y\in W} \Vert R_W^{-1}( b- A u_{r,p} ) - y \Vert_W,$$ so that  $\| R_W^{-1}(b-Au_{r,p}) - y_{r,p}\|_{W} \leq \| R_W^{-1}(b-Au_{r,p}) \|_{W}$.  Using \eqref{eq:idealrieszmap}, it follows
 \begin{equation}\label{eq:tmp5654}
  \| u-u_{r,p} -  R_V^{-1}A^* y_{r,p} \|_{V} \leq \| u-u_{r,p} \|_{V},
 \end{equation}
 Using \eqref{eq:tmp5654} in \eqref{eq:tmp2}, taking the infimum over $y\in T_p$, dividing by $\| v \|_V$ and taking the supremum over $v\in V_r\setminus \{0\}$, we obtain
 \begin{equation*}
  \| u_r^\perp-u_{r,p} \|_V \leq \delta_{V_r,T_p} \| u-u_{r,p}\|.
 \end{equation*}
From  \eqref{deltar_relation1} (with $W_r$ replaced by $T_p$) and  \eqref{eq:idealrieszmap}
(which implies $\| A^*y \|_{V'} = \| y \|_{W}$), we obtain \eqref{deltarp_alpharp}.
Then, from the definition of $u_r^\perp$, we have 
 \begin{equation*}
 \| u-u_{r,p}\|^2 = \| u-u_{r}^*\|^2 + \| u_{r}^* - u_{r,p}\|^2 \le \| u-u_{r}^*\|^2 +  \delta_{V_r,T_p}^2\| u-u_{r,p}\|^2,
  \end{equation*}
from which we deduce \eqref{eq:control_PGD}.
\end{proof}
~\\
From the definition \eqref{eq:delta_VrWp} of $\delta_{V_r ,T_p}$, we easily deduce the following corollary.
\begin{corollary}\label{cor:IdealTestSpace} ~If ~$T_p$ is~ such that~ $R_W^{-1}AV_r \subset T_p$ (or equivalently $V_r \subset R_V^{-1} A^* T_p$) then $\delta_{V_r ,T_p} =0$ and $u_{r,p}$ coincides with the best approximation $u_r^\perp$ of $u$ in $V_r$.
\end{corollary}

\begin{remark}\label{rmk:weakcoercive}
Note that \eqref{eq:control_PGD} and \eqref{deltarp_alpharp} give 
 \begin{equation*}
  \| u - u_{r,p}  \|_V \leq \frac{1}{\alpha_{V_r,T_p}} ~\min_{v\in V_r}\| u-v \|_{V},
 \end{equation*}
 which is sharper than the classical error bound obtained by the Cea's lemma 
 \begin{equation*}  
  \| u - u_{r,p}  \|_V \leq (1+\frac{1}{\alpha_{V_r,T_p}}) ~\min_{v\in V_r}\| u-v \|_{V}.
 \end{equation*}
\end{remark}

Now, we consider the approximation $\widetilde s$ of $s$ defined by 
\begin{equation}\label{eq:def_st_uQrp}
 \widetilde s = Lu_{r,p} + LR_V^{-1}A^* y_{r,p},
\end{equation}
where $(u_{r,p},y_{r,p})\in V_r\times T_p$ is the solution of the saddle point problem \eqref{eq:PGD}. The following proposition provides an error bound for the approximation of the variable of interest.\\

\begin{proposition}\label{prop:bounds_QoI_saddle_point}
The approximation $\widetilde s$ defined by \eqref {eq:def_st_uQrp} satisfies
 \begin{equation}
  \| s-\widetilde s \|_Z \leq \delta_{T_p}^L \| u-u_{r,p} -  R_V^{-1}A^* y_{r,p} \|_V, \label{eq:control_VI_PGD_1} 
 \end{equation}
 with
\begin{align}
  \delta_{T_p}^L = \sup_{0\neq z'\in Z'}\min_{y\in T_p}\frac{\| L^* z' -A^*y \|_{V'}}{\| z' \|_{Z'}},\label{defdeltaL}
\end{align}
and
 \begin{equation}
  \| s-\widetilde s \|_Z \leq \frac{\delta_{T_p}^L}{\sqrt{1-(\delta_{V_r,T_p})^2}} \min_{v\in V_r} \| u-v\|_V. \label{eq:control_VI_PGD_2}
 \end{equation}  

\end{proposition}

\begin{proof}
 For any $z'\in Z'$ and $y\in T_p$, we have
 \begin{align*}
  \langle s - \widetilde s, z' \rangle &\overset{\eqref{eq:def_st_uQrp}}{=} \langle u - u_{r,p}-R_V^{-1}A^* y_{r,p}, L^*  z'\rangle \\
  &\overset{\eqref{eq:PGD_1} \& \eqref{eq:idealrieszmap}}{=} \langle u - u_{r,p}-R_V^{-1}A^* y_{r,p}, L^* z' - A^* y\rangle \\
  &\leq \| u - u_{r,p}-R_V^{-1}A^* y_{r,p} \|_V \| L^* z' - A^* y \|_{V'}.
 \end{align*}
 Taking the minimum over $y\in T_p$, dividing by $\| z' \|_{Z'}$ and taking the supremum over $z'\in Z'\setminus\{0\}$, we obtain \eqref{eq:control_VI_PGD_1}. Finally, thanks to \eqref{eq:control_VI_PGD_1}, \eqref{eq:tmp5654} and \eqref{eq:control_PGD}, we obtain \eqref{eq:control_VI_PGD_2}.
\end{proof}

We observe that $T_p$ impacts both the quality of the projection of $u$ (via the constant $\delta_{V_r,T_p}$ in \eqref{eq:control_PGD}) and the quality of the approximation of the variable of interest (via constants $\delta_{V_r,T_p}$ and $\delta^L_{T_p}$ in \eqref{eq:control_VI_PGD_2}). Then, we will consider for $T_p$ spaces of the form
\begin{equation}\label{eq:def_Wp}
 T_p =  W_r + W^Q_k,
\end{equation}
with $\dim(W_r)=r$. This implies
\begin{equation*}
 \delta_{T_p}^L \leq \delta_{W^Q_k}^L \quad \text{and} \quad \delta_{V_r,T_p} \leq \delta_{V_r, W_r},
\end{equation*}
so that the error bound \eqref{eq:control_VI_PGD_2} for the variable of interest  is better than the error bound \eqref{eq:control_VI_PD_2} of the primal-dual method  with primal approximation space $V_r$, primal test space $ W_r$ and dual approximation space $W_k^{Q}$. Therefore, we expect the approximation $u_{r,p}$ to be closer to the solution $u$ than the Petrov-Galerkin projection $u_r$. Also, the approximation of the quantity of interest is expected to be improved.

\begin{remark}[Symmetric coercive case]\label{rmk:SPD_2}
 Let us consider the case where $A$ is symmetric and coercive, $R_V=R_W = A$ and $W_r=V_r$. The choice \eqref{eq:def_Wp} implies that  $V_r \subset T_p$, so that $T_p$ admits the orthogonal decomposition $T_p = V_r \oplus (T_p \cap V_r^\perp)  $. Equation \eqref{eq:PGD_2} implies that $y_{r,p} \in T_p \cap V_r^\perp$. Let  $t_{r,p}   =y_{r,p}+u_{r,p}\in T_p$.  Equation \eqref{eq:PGD_1} gives $\langle R_V t_{r,p} ,y\rangle = \langle R_V u , y\rangle$ for all $y\in T_p$, which implies that $t_{r,p}$ is the orthogonal projection of $u$ on $T_p$, where $u_{r,p}$ and $y_{r,p}$ are the orthogonal projections of $u$ on $V_r$ and $T_p \cap V_r^\perp$ respectively.  Furthermore, the approximation of the variable of interest \eqref{eq:def_st_uQrp} is given by $ \widetilde s = L t_{r,p}$. We conclude that in this particular setting, 
the saddle point approach can be simply interpreted as an orthogonal projection of $u$ on the enriched space $ T_p = V_r+W_k^Q$, followed by a standard estimation of the variable of interest. 
\end{remark}

%%%%%%%%%%%%%%%%%%%%%%%%%%%%%%%%%%%%%%%%%%%%%%%%%%%%%%%%%%%%%%%%%%%%%%%%%%%%%%%%%%%%%%%%%%%
%%%%%%%%%%%%%%%%%%%%%%%%%%%%  APPLICATION TO MOR     %%%%%%%%%%%%%%%%%%%%%%%%%%%%%%%%%%%%%%
%%%%%%%%%%%%%%%%%%%%%%%%%%%%%%%%%%%%%%%%%%%%%%%%%%%%%%%%%%%%%%%%%%%%%%%%%%%%%%%%%%%%%%%%%%%
\section{Goal-oriented projections for parameter-dependent equations}\label{sec:Application2MOR}

We now consider a parameter-dependent equation 
$ A(\xi)u(\xi)=b(\xi) $ where $\xi$ denotes a parameter taking values in a set $\Xi\subset\mathbb{R}^d$, $A(\xi) \in \mathcal{L}(V,W')$ and $b(\xi) \in W'$. The variable of interest is defined by $s(\xi) = L(\xi)u(\xi)$, with 
$L(\xi) \in \mathcal{L}(V,Z)$.\\

In Section \ref{sec:Analysis}, we have presented different projection methods for the estimation of the variable of interest which rely on the introduction of three spaces: the primal approximation space $V_r$, the primal test space $W_r$ and the dual approximation space $W_k^Q$. We recall that for the saddle point approach, we introduce the space $T_p=W_r + W_k^Q$. We adopt an \emph{offline/online strategy}. Reduced (low-dimensional) spaces $V_r$, $W_r$ and $W_k^Q$  are constructed during the \emph{offline phase}. Then, the projections on these reduced spaces and the evaluations of the variable of interest are rapidly computed for any parameter value $\xi\in\Xi$ during  the \emph{online phase}.\\

In Section \ref{sec:constructionWr}, we will first consider the construction of the test space $W_r$. 
For scalar-valued variables of interest, reduced spaces $V_r$ and $W_k^Q$ are classically defined as the span of snapshots of the primal and dual solutions $u(\xi)$ and $Q(\xi)$. These snapshots can be selected at random, using samples drawn according a certain probability measure over $\Xi$, see \emph{e.g.} \cite{Prudhomme2002}. Another popular method is to select the snapshots in a greedy way \cite{Cuong05,haasdonk2014reduced,Rozza2008}, with a uniform control of the error $\| s(\xi) - \widetilde s(\xi)\|_Z$ over $\Xi$. This method requires an estimation of the error on the variable of interest. In the same lines, we introduce error estimates for vector-valued variables of interest in Section \ref{sec:errorEstimate}, and we propose greedy algorithms for the construction of $V_r$ and $W_k^Q$ in Section \ref{sec:offline}.

\subsection{Construction of the test space $W_r$}\label{sec:constructionWr}

Assuming that the primal approximation space $V_r$ is given, we know from the previous section that $W_r$ should be chosen such that $\delta_{V_r,W_r}$ is as close to zero as possible (see Propositions \ref{prop:1}, \ref{prop:NewBounds}, \ref{prop:bounds_saddle_point} and \ref{prop:bounds_QoI_saddle_point}). 
In the literature, $W_r=V_r$ is a common choice (standard Galerkin projection). 
When the operator $A(\xi)$ is symmetric and coercive, we can choose $W_r=V_r$ which is the optimal test space with respect to the norm induced by $A(\xi)$ (see Remark \ref{rmk:SPD_1}).  However, this choice may lead to an inaccurate projection of the primal variable when the operator is ill-conditioned (i.e. 
$\frac{\beta}{\alpha} \gg1$).
 In the case of non coercive operators, a parameter-dependent test space is generally defined by $W_r=W_r(\xi)=R_W^{-1}A(\xi) V_r$, where $R_W^{-1}A(\xi)$ is called the ``supremizer operator'' (see e.g. {\cite{Rozza2007a,MPR02} }). This  approach is no more than a minimal residual method since the resulting Petrov-Galerkin projection defined by \eqref{eq:def_PG} is $u_r(\xi) = \arg \min_{v_r \in V_r}\| A(\xi)v_r - b(\xi)\|_{W'}$.
 In Section \ref{sec:PG}, we have seen that the Petrov-Galerkin projection with an ideal test space 
 \begin{equation}\label{eq:IdealTestSpace}
 W_r(\xi) = A(\xi)^{-*} R_V {V}_r
\end{equation}
coincides with the best approximation. Having a basis $v_1,\hdots,v_r$ of $V_r$, the computation of this ideal parameter-dependent test space 
 would require the computation of $A^{-*}(\xi) R_V v_i$ for all $1\le i\le r$ for each parameter's value $\xi$, which is unfeasible in practice. 
Up to our knowledge, the only attempt to construct quasi-optimal test spaces for non symmetric and weakly coercive operators can be found in \cite{Dahmen2013}, where the authors proposed a greedy algorithm for the construction of a (parameter independent) test space which ensures the quasi-optimality constant to be uniformly bounded by an arbitrarily small constant. Here, we adopt an alternative approach where the (parameter-dependent) test space is defined by 
\begin{equation}\label{eq:Wr_PrecondVr}
 {W}_r(\xi) = P_m(\xi)^* R_V {V}_r,
\end{equation}
where $P_m(\xi)$ is an interpolation of the inverse of $A(\xi)$ using $m$ interpolation points in the parameter set $\Xi$. In practice, when  $A(\xi)$ is a matrix, algorithms developed in \cite{zahm2015interpolation} can be used. This will be detailed later on. The underlying idea is to obtain a test space as close as possible to the ideal test space $A^{-*}(\xi)R_V {V}_r$ defined in \eqref{eq:IdealTestSpace}. For $m=0$, with $P_0(\xi)=R_V^{-1}$ by convention, we have 
$W_r=V_r$, which yields the standard Galerkin projection.

\subsection{Error estimates for vector-valued variables of interest}\label{sec:errorEstimate}

In this section, we propose practical error estimates for the variable of interest, first for the primal-dual approach and then for the saddle point method.

\subsubsection{Primal-dual approach} 
Given approximations $\widetilde u $ and $\widetilde Q$ of the primal solution $u$ and the dual solution  $Q$ respectively, 
a standard approach is to start from the error bound
\begin{equation*}
 \| s(\xi)-\widetilde s(\xi)\|_Z \leq \| u(\xi)-\widetilde u(\xi)\|_V \| L(\xi)^* - A(\xi)^* \widetilde Q(\xi) \|_{Z'\rightarrow V'},
\end{equation*}
which is provided by Proposition \ref{prop:ClassicalBounds}. This suggests to measure the norm of the residuals associated to the primal and dual variables. In practice, we distinguish two cases.\\

In the case where the operator $A(\xi)$ is symmetric and coercive, it is natural to choose the parameter-dependent norm $\|\cdot\|_V$ as the one induced by the operator, i.e. $R_V=R_V(\xi)=A(\xi)$. However, neither the primal error $\Vert  u(\xi) - \widetilde u(\xi) \Vert_V = \langle b(\xi)- A(\xi) \widetilde u(\xi)  , u(\xi)-\widetilde u(\xi) \rangle $ nor the dual residual norm $\| L(\xi)^* - A(\xi)^* \widetilde Q(\xi) \|_{Z'\rightarrow V'} = \sup_{\Vert z \Vert_Z=1} \Vert  Q(\xi) z -\widetilde Q(\xi) z \Vert_Z$ can be computed without computing  the primal and dual solutions $u(\xi)$ and $Q(\xi)$. The classical way to circumvent this issue is to introduce a parameter-independent norm $\|\cdot\|_{V_0}$, which is in general the ``natural'' norm associated to the space $V$, and to measure  residuals  with the associated dual norm $\Vert \cdot \Vert_{V_0'}$. Here we assume  that the operator $A(\xi)$ satisfies
\begin{equation}
 \alpha(\xi) \| v \|_{V_0} \leq \| A(\xi)v \|_{V_0'} 
 \label{eq:normchoice}
\end{equation}
for all $v \in V$, where $\alpha(\xi)>0$. By definition of the norm $\| \cdot\|_{V}$, we can write
\begin{equation*}
  \| v \|_{V}^2 = \langle A(\xi) v, v \rangle \leq \| A(\xi) v \|_{V_0'} \| v \|_{V_0} \leq \alpha(\xi)^{-1} \| A(\xi) v \|_{V_0'}^2 \quad\quad \forall v\in V.
\end{equation*}
Then we have $\| u(\xi) - \widetilde u(\xi) \|_{V} \leq \alpha(\xi)^{-1/2} \| A(\xi)\widetilde u(\xi) - b(\xi) \|_{V_0'}$. In the same way, we can prove that $\| L(\xi)^* - A(\xi)^* \widetilde Q(\xi) \|_{Z'\rightarrow V'} \leq \alpha(\xi)^{-1/2} \| L(\xi)^* - A(\xi)^* \widetilde Q(\xi) \|_{Z'\rightarrow V_0'}$. Finally, we obtain
\begin{equation}\label{eq:standard_error_estimate}
  \| s(\xi) - \widetilde s(\xi) \|_Z \leq  \frac{\| A(\xi)\widetilde u(\xi) - b(\xi) \|_{V_0'} \| L(\xi)^* - A(\xi)^* \widetilde Q(\xi) \|_{Z'\rightarrow V_0'} }{ \alpha(\xi)} := \Delta(\xi) ,
\end{equation}
where $\Delta(\xi)$ is a certified error bound for the variable of interest, which involves computable primal and dual residual norms. \\

In the general case, we consider for $\| \cdot \|_V$ the natural norm on $V$, i.e. 
 $\|\cdot\|_{V}=\|\cdot\|_{V_0}$. As a consequence, the norm of the dual residual is computable, but the computation of the error $\| u(\xi) - \widetilde u(\xi) \|_{V_0}$ requires the primal solution $u(\xi)$ which is not available in practice. Once again, we assume that the operator satisfies the property \eqref{eq:normchoice} so that we can write $\| u(\xi)-\widetilde u(\xi)\|_{V_0} \leq \alpha(\xi)^{-1}\| A(\xi)\widetilde u(\xi) - b(\xi)\|_{V_0'}$. Then we  end up with the same error bound \eqref{eq:standard_error_estimate} for the variable of interest.

\subsubsection{Saddle point method} 
We now derive new error bounds in the case where the approximation $\widetilde s(\xi)$ is obtained by the saddle point method introduced in Section \ref{sec:SaddlePoint}. Let us start from the error bound
\begin{align}
 \| &s(\xi)-\widetilde s(\xi) \|_Z \nonumber \\
 &\leq \sup_{0\neq z'\in Z'}\min_{y\in T_p}\frac{\| L(\xi)^* z' -A(\xi)^*y \|_{V'}}{\| z' \|_{Z'}} \| u(\xi)-u_{r,p}(\xi) -  R_V^{-1}A(\xi)^* y_{r,p}(\xi) \|_V  \label{eq:tmp5731}
\end{align}
provided by Proposition \ref{prop:bounds_QoI_saddle_point}. Once again, we distinguish two cases.\\

For the case where the operator $A(\xi)$ is symmetric and coercive, we consider for $\|\cdot\|_V$ the norm induced by the operator, i.e. $R_V=A$. According to Remark \ref{rmk:SPD_2}, the quantity $t_{r,p}(\xi)=u_{r,p}(\xi) -  R_V^{-1}(\xi)A(\xi)^* y_{r,p}(\xi)=u_{r,p}(\xi)+y_{r,p}(\xi)$ is nothing but the orthogonal projection of $u(\xi)$ onto $T_p = W_r + W_k^Q$, with $W_r=V_r$. Then for any $\widetilde t_{r,p}\in T_p$ we have 
 \begin{equation*}
  \| u(\xi) - t_{r,p}(\xi) \|_{V}^2 \leq \| u(\xi) - \widetilde t_{r,p} \|_{V}^2 \leq \alpha(\xi)^{-1} \| b(\xi) - A(\xi) \widetilde t_{r,p}  \|_{V_0'}^2,
 \end{equation*}
 where the norm $\|\cdot\|_{V_0}$ is the natural norm on $V$ such that \eqref{eq:normchoice} holds. Then, taking the infimum over $\widetilde t_{r,p}\in T_p$ we obtain 
 \begin{equation*}
  \| u(\xi) - t_{r,p}(\xi) \|_{V}\leq \alpha(\xi)^{-1/2}\inf_{\widetilde t_{r,p}\in T_p}\| A(\xi) \widetilde t_{r,p} - b(\xi) \|_{V_0'}.
 \end{equation*}
 Finally, we obtain that 
 \begin{align}
  \| &s(\xi) - \widetilde s(\xi) \|_Z \nonumber \\
  &\leq \frac{1}{\alpha(\xi)}\sup_{0\neq z'\in Z'}\min_{y\in T_p}\frac{\| L(\xi)^* z' -A(\xi)^*y \|_{V_0'}}{\| z' \|_{Z'}} \min_{\widetilde t_{r,p}\in T_p}\| A(\xi) \widetilde t_{r,p} - b(\xi) \|_{V_0'} := \Delta(\xi). \label{eq:error_estimate_SPD_SaddlePoint}
 \end{align}
Note that the main difference between this error estimate and the previous one \eqref{eq:standard_error_estimate} is the minimization problem over $T_p$ in both primal and dual residuals. {The solution of those minimization problems lead to additional computational costs, but sharper error bounds will be obtained, as illustrated by the numerical examples in the next section.}\\
 
For the general case, we consider $\|\cdot\|_V = \|\cdot\|_{V_0}$. {Starting from \eqref{eq:tmp5731} and using the relation \eqref{eq:normchoice} to bound the primal error by the primal residual norm,} we obtain the following error estimate
 \begin{align}\label{eq:error_estimate_General_SaddlePoint}
  \| &s(\xi) - \widetilde s(\xi) \|_Z \\
  &\leq \frac{1}{\alpha(\xi)}\sup_{0\neq z'\in Z'}\min_{y\in T_p}\frac{\| L(\xi)^* z' -A(\xi)^*y \|_{V_0'}}{\| z' \|_{Z'}}  \| A(\xi)t_{r,p}(\xi) -b(\xi) \|_{V_0'} 
:=  \Delta(\xi) , \nonumber
 \end{align}
 where $t_{r,p}(\xi)= u_{r,p}(\xi) +  R_{V_0}^{-1}A(\xi)^* y_{r,p}(\xi)$.\\

\begin{remark}
All the proposed error estimates rely on the knowledge of $\alpha(\xi)$. In the case where $\alpha(\xi)$ can not be easily computed, we can replace it by a lower bound $\alpha^{LB}(\xi)\leq \alpha(\xi)$, e.g. provided by a SCM procedure \cite{Huynh2007c}. This option will not be considered here. Another option is to remove $\alpha(\xi)$ from the definitions of $\Delta(\xi)$, therefore leading to error estimates which are no more certified error bounds. 
\end{remark}

\subsection{Greedy construction of the reduced spaces}\label{sec:offline}

Here, we propose different greedy algorithms for the construction of the reduced spaces $V_r$ and $W_k^Q$. At each iteration, we search for a parameter value $\xi^* \in \Xi$ where the error estimate $\Delta(\xi)$ is maximum, i.e. 
\begin{equation}\label{eq:MaxDelta}
{ \xi^* \in \arg\max_{\xi\in\Xi} \Delta(\xi)}.
\end{equation}
{
A first strategy is to simultaneously enrich both the primal approximation space
\begin{equation}\label{eq:Enrich_Vr}
 V_{r+1} = V_r + \text{span}( u(\xi^*) )
\end{equation}
and the dual approximation space
\begin{equation}\label{eq:Enrich_Wk}
 W_{k+l}^Q = W_k^{Q} + \text{range}( Q(\xi^*) )
\end{equation}
at each iteration. This strategy is referred as the \emph{simultaneous construction}, as opposed to the \emph{alternate construction} which consists in enriching $W_k^Q$ (resp. $V_r$) if $V_r$ (resp. $W_k^Q$) were enriched at the previous greedy iteration step. 
}

\begin{remark}
 In the literature, and for scalar-valued variables of interest, the classical approaches are either a \emph{separated construction} of $V_r$ and $W_{k}^Q$ (using two independent greedy algorithms, see for e.g. \cite{Grepl2005,Rozza2008}), or a \emph{simultaneous construction} (see \emph{e.g.} \cite{Quarteroni2011}). The latter option can take advantage of a single factorization of the operator $A(\xi^*)$ to compute both the primal and dual variables. The \emph{alternate construction} proposed here is not usual. This possibility is mentioned in remark 2.47 of the tutorial \cite{haasdonk2014reduced}.
\end{remark}

For vector-valued variables of interest ($\text{dim}(Z)>1$), 
the enrichment strategy \eqref{eq:Enrich_Wk} makes sense only if $\text{dim}(\text{range}( Q(\xi^*) )<\infty$, in which case $l = \text{dim}(W_{k+l}^Q) - \text{dim}(W_{l}^Q) <\infty$. However,  if $\text{dim}(\text{range}( Q(\xi^*) )$ is finite but very high, the enrichment strategy 
\eqref{eq:Enrich_Wk} may lead to a rapid increase of the dimension of the dual approximation space. 
Therefore, when  $\text{dim}(\text{range} (Q(\xi^*) ))$ is infinite or very high, we propose to replace 
the enrichment strategy \eqref{eq:Enrich_Wk} by
\begin{equation}\label{eq:Enrich_Wk_one}
 W_{k+1}^Q = W_k^{Q} + \text{span}( Q(\xi^*)z' ),
\end{equation}
where the space $W_k^{Q}$ is enriched with a single vector $Q(\xi^*)z'$, with 
 $z'\in Z'$ such that
\begin{align}
 &z' \in \arg\max_{\widetilde z'\in Z'} \frac{\| ( L(\xi^*)^*- A(\xi^*)^* \widetilde Q(\xi))\widetilde z' \|_{V_0'}}{\| \widetilde z' \|_{Z'}} \quad\text{(for {primal-dual} method), or} \label{eq:Find_z_primaldual}\\
 &z' \in \arg\max_{\widetilde z'\in Z'} \min_{y\in W_k^{Q}} \frac{\| L(\xi^*)^*\widetilde z' - A(\xi^*)^* y \|_{V_0'}}{\| \widetilde z' \|_{Z'}} \quad\text{(for {saddle point} method).}\label{eq:Find_z_saddle}
\end{align}
Contrarily to the {full} enrichment \eqref{eq:Enrich_Wk}, this {partial} enrichment does not necessarily lead to a zero error at the point $\xi^*$ for the next iterations. Then we expect that \eqref{eq:Enrich_Wk_one} will deteriorate the convergence properties of the algorithm, but for $\text{dim}(Z)\gg 1$, the space $W_{k+1}^Q$ defined by \eqref{eq:Enrich_Wk_one} will have a much lower dimension than the space  $ W_{k+l}^Q$ defined by \eqref{eq:Enrich_Wk}. It is worth mentioning that in \cite{Dahmen2013}, the authors propose a similar partial enrichment strategy for the test space  $T_p$ but not in a goal-oriented framework. 

The definition \eqref{eq:Wr_PrecondVr} of the test space $W_r$ requires the definition of a preconditioner $P_m(\xi)$ which is here constructed by interpolation of the inverse of $A(\xi)$. Following the idea of \cite{zahm2015interpolation}, the interpolation points for the preconditioner are chosen as the points where solutions (primal and dual) have already been computed, i.e. the points given by \eqref{eq:MaxDelta}. The resulting algorithms are summarized in Algorithm \ref{alg:Simultaneous} and Algorithm \ref{alg:Alternated} respectively for the simultaneous and the alternate constructions of $V_r$ and $W_k^Q$.

\begin{algorithm}[ht!]
\begin{algorithmic}[1]
\REQUIRE Error estimator $\Delta(\cdot)$, a samples set $\Xi$, maximum iteration $I$
\STATE Initialize $r,k=0$ and the spaces $V_r=\{0\}$ and $W_k^Q=\{0\}$
\FOR{$i=1$ to $I$}
\STATE Find $\xi_i \in \arg\max_{\xi\in \Xi} \Delta(\xi)$
\STATE Compute a factorization of $A(\xi_i)$ and update the preconditioner if needed
\STATE Solve $u(\xi_i)=A(\xi_i)^{-1}b(\xi_i)$ 
\STATE Update $V_{r+1}=V_r+\text{span}(u(\xi_i))$, and $r\leftarrow r+1$
\IF{Full dual enrichment}
\STATE Solve $Q(\xi_i)=A(\xi_i)^{-*}L(\xi_i)^*$
\STATE Update $W_{k+l}^Q=W_{k}^Q + \text{range}(Q(\xi_i))$, and $k\leftarrow k+l$
\ELSIF{Partial dual enrichment}
\STATE Find $z'$ according to \eqref{eq:Find_z_saddle} or \eqref{eq:Find_z_primaldual}
\STATE Solve $y(\xi_i)=A(\xi_i)^{-*}(L(\xi_i)^*z')$ 
\STATE Update $W_{k+1}^Q = W_{k}^Q + \text{span}(y(\xi_i))$, and $k\leftarrow k+1$
\ENDIF
\ENDFOR
\end{algorithmic}
\caption{Simultaneous construction of $V_r$ and $W_k^Q$ }
\label{alg:Simultaneous}
\end{algorithm}

\begin{algorithm}[ht!]
\begin{algorithmic}[1]
\REQUIRE Error estimator $\Delta(\cdot)$, a samples set $\Xi$, maximum iteration $I$
\STATE Initialize $r,k=0$ and the spaces $V_r=\{0\}$ and $W_k^Q=\{0\}$
\FOR{$i=1$ to $I$}
\STATE Find $\xi_i \in \arg\max_{\xi\in \Xi} \Delta(\xi)$
\STATE Compute a factorization of $A(\xi_i)$ and update the preconditioner if needed
\IF{$i$ is even}
\STATE Solve $u(\xi_i)=A(\xi_i)^{-1}b(\xi_i)$ 
\STATE Update $V_{r+1}=V_r+\text{span}(u(\xi_i))$, and $r\leftarrow r+1$
\ELSIF{$i$ is odd}
\IF{Full dual enrichment}
\STATE Solve $Q(\xi_i)=A(\xi_i)^{-*}L(\xi_i)^*$ 
\STATE Update $W_{k+l}^Q=W_{k}^Q + \text{range}(Q(\xi_i))$, and $k\leftarrow k+l$
\ELSIF{Partial dual enrichment}
\STATE Find $z'$ according to \eqref{eq:Find_z_saddle} or \eqref{eq:Find_z_primaldual}
\STATE Solve $y(\xi_i)=A(\xi_i)^{-*}(L(\xi_i)^*z')$ 
\STATE Update $W_{k+1}^Q = W_{k}^Q + \text{span}(y(\xi_i))$, and $k\leftarrow k+1$
\ENDIF
\ENDIF
\ENDFOR
\end{algorithmic}
\caption{Alternate construction of $V_r$ and $W_k^Q$ }
\label{alg:Alternated}
\end{algorithm}

\section{Numerical results}\label{sec:Numerical}

In this section, we present numerical applications of the methods proposed in Sections \ref{sec:Analysis} and \ref{sec:Application2MOR}. We first describe the applications in
Section \ref{applis}. Then we compare the projection methods for the estimation of a variable of interest in Section \ref{compproj}. Finally, we study the behavior of the proposed greedy algorithms for the construction of the reduced spaces in  Section \ref{appligreedy}.

\subsection{Applications}\label{applis}

\subsubsection{Application 1 : a symmetric problem} \label{sec:Bridge}

We consider a linear elasticity problem\footnote{The authors thank Mathilde Chevreuil for having proposed this benchmark problem.} $\text{div}(K(\xi):\varepsilon(\mathbf{u}(\xi))) = 0$ over a domain $\Omega$ (represented in Figure \ref{fig:PONT_settings_a}), where $\mathbf{u}(\xi):\Omega \to\mathbb{R}^3$ is the displacement field and  $\varepsilon(\mathbf{u})=\frac{1}{2}(\nabla \mathbf{u} +\nabla \mathbf{u}^T  )\in \mathbb{R}^{3\times 3}$ is the strain tensor associated to the displacement field $ \mathbf{u}$. The Hooke tensor $K(\xi)$ is such that
\begin{equation*}
 K(\xi):\varepsilon(\mathbf{u}(\xi))= \frac{E(\xi)}{1+\nu}\Big(\varepsilon(\mathbf{u}(\xi))+\frac{\nu }{1-2\nu}\text{trace}(\varepsilon(\mathbf{u}(\xi))) I_3\Big),
\end{equation*}
where $\nu=0.3$ is the Poisson coefficient and $E(\xi)$ is the Young modulus defined by $E(\xi) = 1_{\Omega_0} + \sum_{i=1}^6\xi_i 1_{\Omega_i}$, $1_{\Omega_i}$ being the indicator function of the subdomain $\Omega_i$, see Figure \ref{fig:PONT_settings_b}. The components of $\xi=(\xi_1,\hdots,\xi_6)$ are independent and log-uniformly distributed over $[10^{-1},10]$. We impose homogeneous Dirichlet boundary condition $\mathbf{u}(\xi)=0$ on  $\Gamma_{D}$ (red lines), a unit vertical surface load on $\Gamma_{L}$ (green square), and a zero surface load on the complementary part of the boundary (see Figure \ref{fig:PONT_settings_a}). We consider the Galerkin approximation $\mathbf{u}^h(\xi)$ of $\mathbf{u}(\xi)$ on a $\mathbb{P}_1$ finite element approximation space $\mathbf{V}^h=\text{span}(\boldsymbol{\phi}_i)_{i=1}^n\subset \{\mathbf{v}\in H^1(\Omega)^3: \mathbf{v}_{|\Gamma_D}=0\}$ of dimension $n=8916$ associated to the mesh plotted in Figure \ref{fig:PONT_settings_b}. The vector $u(\xi)\in V=\mathbb{R}^n$ such that $\mathbf{u}^h(\xi)=\sum_{i=1}^n u_i(\xi)\boldsymbol{\phi}_i$ is the solution of the linear system $A(\xi)u(\xi)=b$ of size $n$, with
\begin{align}\label{eq:BRIDGE_Aaffine}
 A(\xi) &= A^{(0)} + \sum_{k=1}^6 \xi_i A^{(k)} ~,\quad
 A^{(k)}_{i,j} = \int_{\Omega_k} \nabla \boldsymbol{\phi}_i : K_0 : \nabla \boldsymbol{\phi}_j ~\text{d}\Omega,
\end{align}
and $b_i = \int_{\Gamma_{L}} -\mathbf{e}_3\cdot \boldsymbol{\phi}_i ~\text{d}\Gamma$, where $K_0$ denotes the Hooke tensor with the Young modulus $E=1$. The norm $\| \cdot\|_{V}$ on the space $V$ is chosen such that  $\| v \|_{V}^2 = \langle A(\xi) v,v \rangle$, that means $R_V=A(\xi)$.
We also consider the parameter-independent norm $\|\cdot\|_{V_0}$ defined by 
$\| v \|_{V_0}^2 = \langle A(\xi_0) v,v \rangle$ with 
 $\xi_0=(1,\hdots,1)$. It corresponds to the norm induced by the operator associated with the Hooke tensor $K_0$ instead of $K(\xi)$.\\

Let us consider $s^h(\xi)= \mathbf{u}^h_{|\Gamma}(\xi) \cdot \mathbf{e}_3$ which is the vertical displacement of the Galerkin approximation on the blue line $\Gamma$, see Figure \ref{fig:PONT_settings_a}. We can write $s^h(\xi)=\sum_{j=1}^l s_j(\xi) \psi_j$ where $\{\psi_j\}_{j=1}^l$ is a basis of the space $\{ \mathbf{v}^h_{|\Gamma} \cdot \mathbf{e}_3: \mathbf{v}^h\in \mathbf{V}^h\}$ of dimension $l=44$. Then there exists $L\in\mathbb{R}^{l\times n}$ such that
\begin{equation*}
 s(\xi)=Lu(\xi),
\end{equation*}
where $s(\xi)=(s_1(\xi),\hdots,s_l(\xi))\in Z=\mathbb{R}^l$ is the variable of interest. The norm $\|\cdot\|_Z$ is defined as the canonical norm of $\mathbb{R}^l$.
 
\begin{figure}[h!]
   \centering
   \subfigure[Geometry, boundary condition and variable of interest.]{\centering
     \includegraphics[width=.7\textwidth]{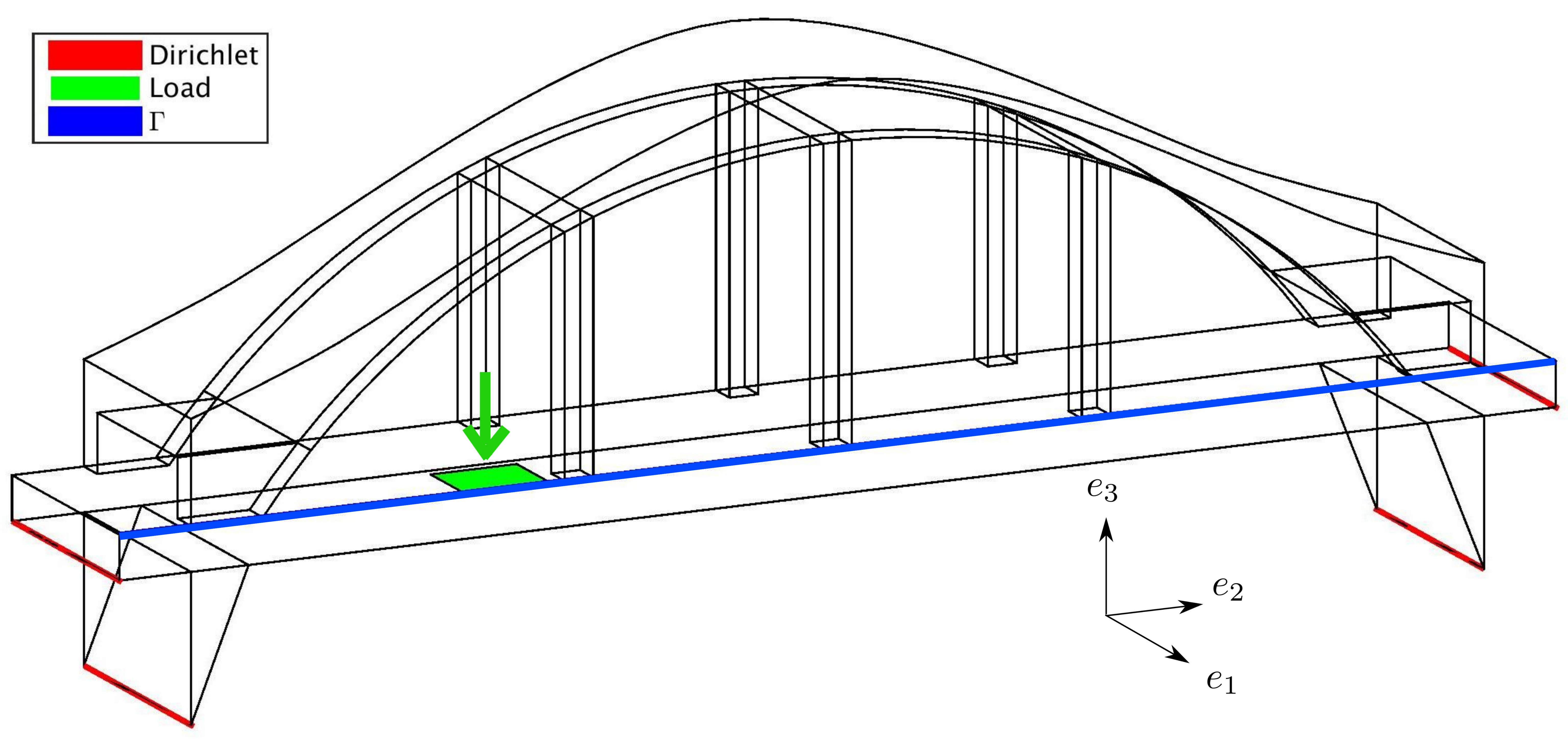}
     \label{fig:PONT_settings_a}
   } \\
   \subfigure[Realization of a solution and mesh of the domain $\Omega$. The colors corresponds to the different sub-domains $\Omega_i$ for $i=0,\hdots,6$.]{\centering
     \includegraphics[width=.7\textwidth]{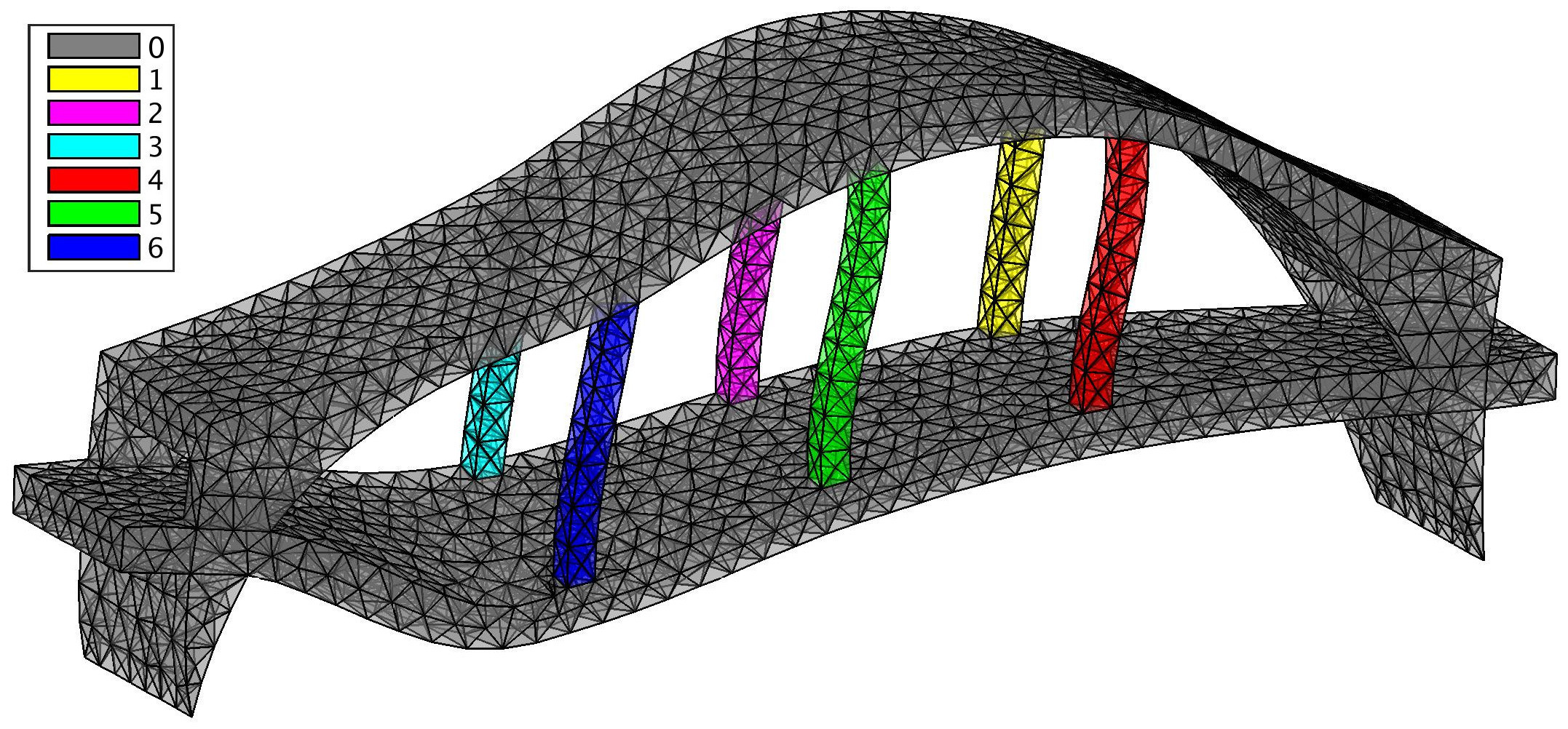}
     \label{fig:PONT_settings_b}
   }
   \caption{Application 1: schematic representation of the problem and a realization of the solution.}
   \label{fig:PONT_settings}
\end{figure}

\subsubsection{Application 2: a non symmetric problem}

We consider the benchmark problem of the cooling of electronic components proposed in the OPUS  project\footnote{See http://www.opus-project.fr}. The equation to solve is an advection-diffusion equation
over  the domain $\Omega \subset \mathbb{R}^2$
\begin{equation}
\label{eq:OPUS}
-\nabla \cdot (\kappa(\xi) \nabla {T}(\xi)) + \mathbf{c}(\xi) \cdot  \nabla  T(\xi)= f,
\end{equation}
whose solution $T(\xi):\Omega \to\mathbb{R}$ is the temperature field. Here $\kappa$ and $\mathbf{c}$ denote respectively  the diffusion coefficient and the advection field, which are parameter-dependent coefficients of the operator.  The full description of this problem is given in \cite{zahm2015interpolation}.  
Here, we only focus on the resulting algebraic parameter-dependent equation coming from stabilized finite element discretization of \eqref{eq:OPUS}, that is $A(\xi)u(\xi)=b(\xi)$, where $u(\xi) \in \mathbb{R}^n$ are the coefficients of the  finite element approximation $T^h = \sum_{i=1}^n u_i \varphi_i$ of $T$, and  where $\xi=(\xi_1,\hdots,\xi_4)$ is a 4-dimensional random vector.
The space $V=\mathbb{R}^n$ with $n=2.8\times 10^4$ is endowed with the norm $\|\cdot\|_{V}=\|\cdot\|_{V_0}$ which corresponds to the $H^1(\Omega)$-norm\footnote{ It means that $\| v \|_{V_0}=\| v^h \|_{H^1(\Omega)}$ for all $v\in V$, where $v^h=\sum_{i=1}^n v_i\varphi_i$.}. The variable of interest $s(\xi) = (s_1(\xi),s_2(\xi))$ is the mean temperature of both  electronic components, with
\begin{equation}
 s_1(\xi) = \frac{1}{|\Omega_{IC_1}|}\int_{\Omega_{IC_1}} T^h(\xi) \text{d}\Omega \quad,\quad s_2(\xi) = \frac{1}{|\Omega_{IC_2}|}\int_{\Omega_{IC_2}} T^h(\xi) \text{d}\Omega, 
\end{equation}
where $\Omega_{IC_i}$ ($i=1,2$) are two subdomains of $\Omega \subset \mathbb{R}^2$ (see \cite[Fig.7]{zahm2015interpolation}). 
Then we can write $s(\xi) = Lu(\xi)$ for an appropriate $L\in\mathbb{R}^{l\times n}$, with $l=2$. Here we have $Z=\mathbb{R}^2$, which we equip with the canonical norm on $\mathbb{R}^2$.

\subsection{Comparison of the projections methods}\label{compproj}

The goal of this section is to compare the projection methods proposed in Section \ref{sec:Analysis} for the estimation of $s(\xi)$. Here the approximation spaces $V_r$, $W_k^Q$ and the test space $W_r$ are given. We denote by $\MV_r$, $\MW_k^Q$ and $\MW_r$ the matrices containing the basis vectors of the corresponding subspaces. {In order to improve condition numbers of reduced systems of equations, these bases are orthogonalized using a Gram-Schmidt procedure.}

\subsubsection{Application 1}

We first detail how we build $\MV_r$, $\MW_k^Q$ and $\MW_r$. The matrix ${\MV}_r$ contains $r=20$ snapshots of the solution: ${\MV}_r=(u(\xi_1),\hdots,u(\xi_{20}))$. The test space is $W_r=V_r$, which corresponds to a standard Galerkin projection method. The matrix ${\MW}_k^Q$ contains $2$ snapshots of the dual variable $Q(\xi) = A(\xi)^{-1}L^* \in\mathbb{R}^{n\times l}$. Then $k=2l=88$. Finally, according to \eqref{eq:def_Wp} the matrix ${\MT}_p=\big( {\MW}_r , {\MW}_k^Q \big)$ is the concatenation of the matrices ${\MW}_r$ and ${\MW}_k^Q$.

We consider a samples set $\Xi_t\subset \Xi$ of size $t=10^4$. For each $\xi\in\Xi_t$ we compute the exact quantity of interest $s(\xi)$ and the approximation $\widetilde s(\xi)$ by the following methods.

\begin{itemize}
 \item \emph{Primal only}: solve the linear system $\big( {\MV}_r^T A(\xi) {\MV}_r\big){U}_r(\xi) = \big({\MV}_r^Tb\big)$ of size $r$ 
 and compute $\widetilde s(\xi) = \big(L{\MV}_r\big) {U}_r(\xi)$.
 \item \emph{Dual only}: solve the linear system $\big(({\MW}_k^Q)^T A(\xi) {\MW}_k^Q \big) {Y}_k(\xi) = \big(({\MW}_k^Q)^Tb \big)$ of size $k$ and compute $\widetilde s(\xi) = \big( L{\MW}_k^Q \big) {Y}_k(\xi)$.\footnote{The dual only method corresponds to the \emph{primal-dual} method where we consider a zero primal approximation, i.e. $V_r=W_r=\{0\}$.}
 \item \emph{Primal-dual}: solve the linear system of the \emph{Primal only} method, solve the linear system $\big( ({\MW}_k^Q)^T A(\xi) {\MW}_k^Q \big) {Y}_k(\xi) = \big( ({\MW}_k^Q)^Tb\big)-\big( ({\MW}_k^Q)^T A(\xi) {\MV}_r \big) {U}_r(\xi)$ of size $k$ and compute $\widetilde s(\xi) = \big( L{\MV}_r\big) {U}_r(\xi) + \big( L{\MW}_k^Q\big) {Y}_k(\xi)$.
 \item \emph{Saddle point}: According to Remark \ref{rmk:SPD_2}, solve the linear system 
 \begin{equation*}
 \big( {\MT}_p^TA(\xi) {\MT}_p\big)  {Y}_p(\xi) = \big( {\MT}_p^Tb \big) 
 \end{equation*}
 of size $p=k+r$, and compute $\widetilde s(\xi) = \big( L {\MT}_p \big) {Y}_{p}(\xi)$.
\end{itemize}

The affine decomposition \eqref{eq:BRIDGE_Aaffine} of matrix $A(\xi)$ allows for a rapid solution of the reduced systems for any parameter $\xi$. \\

Figure \ref{fig:PONT_proj} gives the probability density function (PDF), the $L^\infty$ norm and $L^2$ norm of the error $\| s(\xi)-\widetilde s(\xi)\|_Z$ estimated over the samples set $\Xi_t$. We see that the \emph{primal-dual} method provides errors for the quantity of interest which correspond to the product of the errors of the \emph{primal only} and \emph{dual only} methods. This  reflects the ``squared effect''. Moreover the \emph{saddle point} method provides errors that are almost 10 times lower than the \emph{primal-dual} method. This impressive improvement can be explained by the fact that the proposed problem is ``almost compliant'', in the sense that the primal and dual solutions are similar: the primal solution is associated to a vertical force on the green square of Figure \ref{fig:PONT_settings_a}, and the dual solution is associated to a vertical loading on $\Gamma$. To illustrate this, let us consider a ``less compliant'' application where the variable of interest is defined as the horizontal displacement (in the direction $e_2$, see figure \ref{fig:PONT_settings_a}) of the solution on the blue line $\Gamma$, i.e. $ s^{h}(\xi) = \mathbf{u}_{|\Gamma}(\xi) \cdot \mathbf{e}_{2}$ (instead of $ s^{h}(\xi) = \mathbf{u}_{|\Gamma}(\xi) \cdot \mathbf{e}_{3}$). The results are given in Figure \ref{fig:PONT_proj_L2}. For this new setting, we can draw similar conclusions but the \emph{saddle point} method provides a solution which is ``only'' 2 times better (instead of 10 times) than the \emph{primal-dual} method.

\begin{figure}[h!]
   \centering
   \subfigure[PDF of the error.]{\centering
     \includegraphics[width=.45\textwidth]{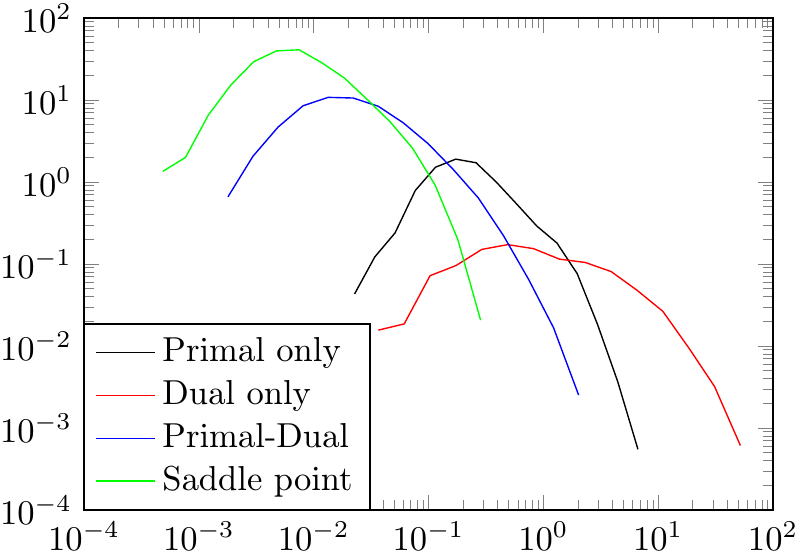}
     \label{fig:PONT_proj_a}
   }~~
   \subfigure[ $L^\infty$ and $L^2$ norm of the error.]{
     \centering
     \includegraphics[scale=1]{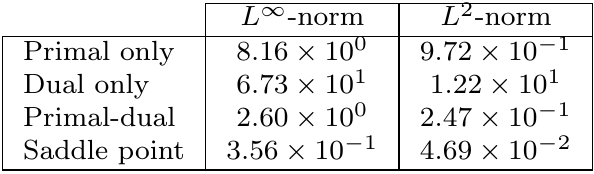}
     \label{fig:PONT_proj_b}
   }
   \caption{Application 1: Probability density function, $L^\infty$ norm and $L^2$ norm of the error $\| s(\xi)-\widetilde s(\xi)\|_Z$ estimated on a samples set of size $10^4$.}
   \label{fig:PONT_proj}
\end{figure}

\begin{figure}[h!]
   \centering
   \subfigure[PDF of the error.]{\centering
     \includegraphics[width=.45\textwidth]{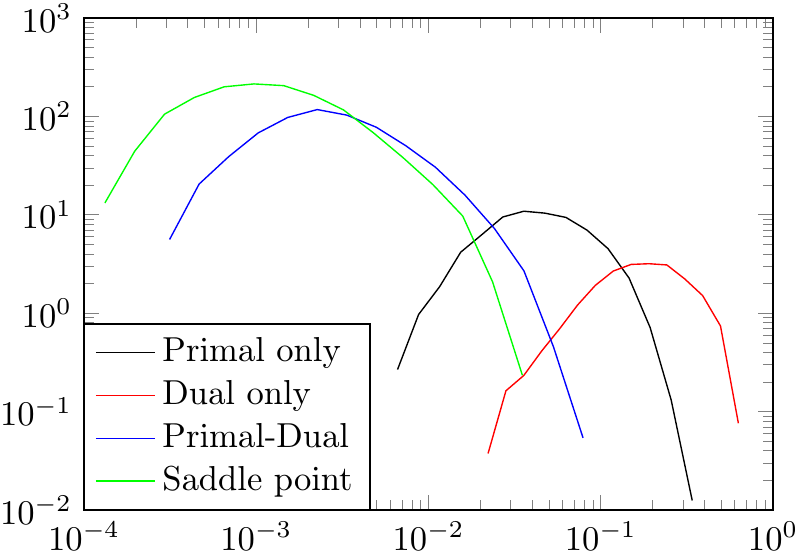}
     \label{fig:PONT_proj_a_L2}
   }~~
   \subfigure[ $L^\infty$ and $L^2$ norm of the error.]{
     \centering
     \includegraphics[scale=1]{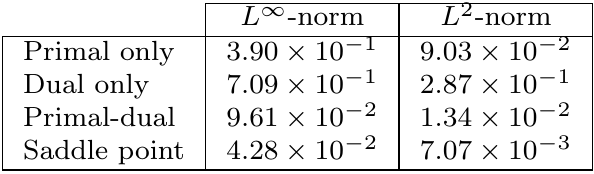}
     \label{fig:PONT_proj_b_L2}
   }
   \caption{Application 1 with a different variable of interest (``less compliant'' case): Probability density function, $L^\infty$ norm and $L^2$ norm of the error $\| s(\xi)-\widetilde s(\xi)\|_Z$ estimated on a samples set of size $10^4$.}
   \label{fig:PONT_proj_L2}
\end{figure}

Now we consider  the effectivity index $\eta(\xi) = \Delta(\xi)/\| s(\xi)-\widetilde s(\xi)\|_Z$ associated to the primal-dual error estimate defined by \eqref{eq:standard_error_estimate} and to the saddle-point error estimate defined by \eqref{eq:error_estimate_SPD_SaddlePoint}. For the considered application, the coercivity constant $\alpha(\xi)$ can be obtained  by the \emph{min-theta} method \cite[Proposition 2.35]{haasdonk2014reduced}. Figure \ref{fig:PONT_proj_ErrorEst} presents statistical information on $\eta(\xi)$: the PDF, the mean, the max-min ratio and the normalized standard deviation estimated on a samples set of size $10^4$. We first observe in Figure \ref{fig:PONT_proj_ErrorEst_a} that the effectivity index is always greater than $1$: this illustrates the fact that the error estimates are certified. Moreover, the error estimate of the saddle point method is much better than the one of the primal-dual method. The max-min ratio and the standard deviation of the corresponding effectivity index are much smaller and  the mean value is much closer to one for the saddle point method.

\begin{figure}[h!]
   \centering
   \subfigure[PDF of $\eta(\xi)$ for the primal-dual method and the saddle point method.]{\centering
     \includegraphics[width=.45\textwidth]{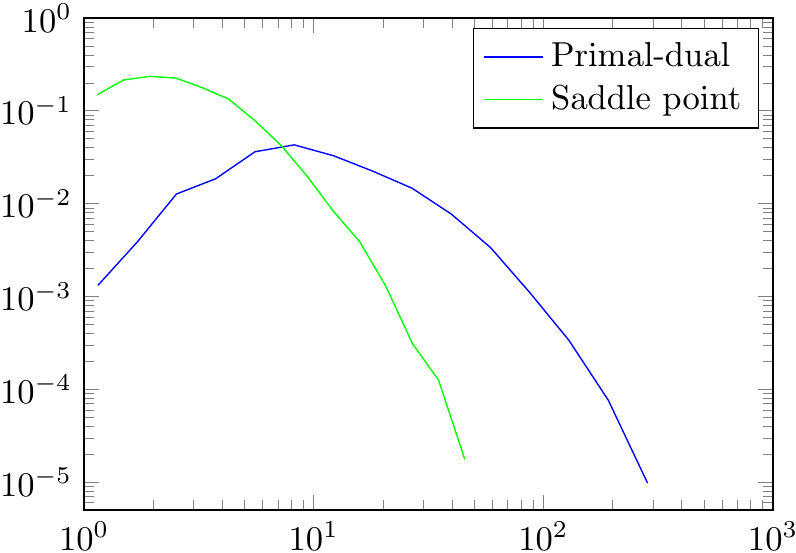}
     \label{fig:PONT_proj_ErrorEst_a}
   }~~
   \subfigure[Statistics of the effectivity index $\eta(\xi)$ for the primal-dual method and saddle point method.]{
     \centering
     \includegraphics[scale=1]{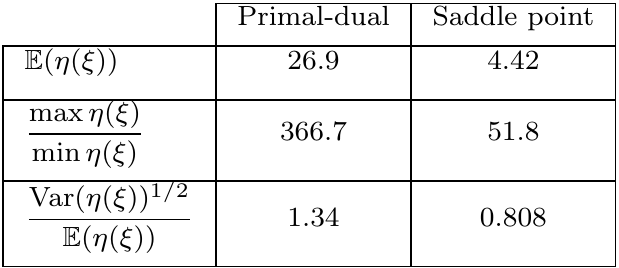}
     \label{fig:PONT_proj_ErrorEst_b}
   }
   \caption{Application 1: Probability density function, mean, min-max ratio and normalized standard deviation of the effectivity index $\eta(\xi) = \Delta(\xi)/\| s(\xi)-\widetilde s(\xi)\|_Z$ estimated on a samples set of size $10^4$. Here, $\Delta(\xi)$ is defined by \eqref{eq:standard_error_estimate} for the primal-dual method and by \eqref{eq:error_estimate_SPD_SaddlePoint} for the saddle point method.}
   \label{fig:PONT_proj_ErrorEst}
\end{figure}

\subsubsection{Application 2}\label{sec:appli22}

For this second application, ${\MV}_r = (u(\xi_1),\hdots,u(\xi_{50}))$ contains $50$ snapshots of the primal solution ($r=50$), and ${\MW}_k^Q=(Q(\xi_1) \hdots Q(\xi_{25}))$ contains $25$ snapshots of the dual solution so that the dimension of $W_k^Q$ is $k=25l=50$. The test space $W_r$ is defined according to \eqref{eq:Wr_PrecondVr}, where $P_m(\xi)$ is an interpolation of $A(\xi)^{-1}$ using $m$ interpolation points selected by a greedy procedure based on the residual $\| I-P_m(\xi)A(\xi)\|_F$ (where $\|\cdot\|_F$ denotes the matrix Frobenius norm), see \cite{zahm2015interpolation}. The interpolation is defined by a Frobenius semi-norm projection (with positivity constraint) using a {random}  matrix with $400$ columns. The matrix 
%P-SRHT
associated to the test space is given by ${\MW}_r(\xi)=P_m^T(\xi)R_V {\MV}_r$.

Once again, we consider a samples set $\Xi_t$ of size $t=10^4$. For any $\xi\in\Xi_t$ we compute the exact quantity of interest $s(\xi)$ and the approximation $\widetilde s(\xi)$ by the following methods.

\begin{itemize}
 \item \emph{Primal only}: solve the linear system $\big({\MW}_r^T(\xi) A(\xi) {\MV}_r\big)U_r(\xi) = {\MW}_r(\xi)^T b$ of size $r$ and compute $\widetilde s(\xi) = \big(L{\MV}_r\big) U_r(\xi)$.
 \item \emph{Dual only}: solve the linear system 
 \begin{equation*}
  \big(({\MW}_k^Q)^T A(\xi) R_V^{-1}A(\xi)^* {\MW}_k^Q \big) Y_k(\xi) = ({\MW}_k^Q)^T b  
 \end{equation*}
 of size $k$ and compute $\widetilde s(\xi) = \big( L R_V^{-1}A(\xi)^* {\MW}_k^Q \big) Y_k(\xi)$.
 \item \emph{Primal-dual}: solve the linear system of the {Primal only} method, solve the linear system 
 \begin{equation*}
  \big( ({\MW}_k^Q)^T A(\xi) R_V^{-1} A(\xi)^* {\MW}_k^Q \big) Y_k(\xi) = \big( ({\MW}_k^Q)^Tb\big) - \big( ({\MW}_k^Q)^T A(\xi) {\MV}_r \big) U_r(\xi)
 \end{equation*}
 of size $k$, and compute $\widetilde s(\xi) = \big( L{\MV}_r\big) U_r(\xi) + \big( L R_V^{-1}A(\xi)^* {\MW}_k^Q \big) Y_k(\xi)$.
 \item \emph{Saddle point}: solve the linear system of size $p+r$
 \begin{equation*}
  \begin{pmatrix} {\MT}_p^T(\xi)A(\xi)R_V^{-1}A(\xi)^* {\MT}_p(\xi) & {\MT}_p^T(\xi)A(\xi){\MV}_r \\ \big({\MT}_p^T(\xi)A(\xi){\MV}_r\big)^T& 0\end{pmatrix}
  \begin{pmatrix} Y_{r,p}(\xi)\\U_{r,p}(\xi) \end{pmatrix}
  =\begin{pmatrix} {\MT}_p(\xi)^Tb  \\0 \end{pmatrix}
 \end{equation*}
 with ${\MT}_p(\xi)=\big( {\MW}_r(\xi),{\MW}_k^Q \big)$, and compute
 \begin{equation*}
  \widetilde s(\xi) = \big(L {\MV}_r\big) U_{r,p}(\xi) + \big(LR_V^{-1} A(\xi)^* {\MT}_p(\xi)\big) Y_{r,p}(\xi).
 \end{equation*}
\end{itemize}

~\\
The numerical results are given in Figure \ref{fig:OPUS_proj}. Once again, the {saddle point} method leads to the lowest error on the variable of interest. Also, we see that a good preconditioner (for example with $m=30$) improves the accuracy for the {saddle point} method, the {primal only} method and the {primal-dual} method. However, this improvement is not really significant for the considered application: the errors are barely divided by $2$ compared to the non preconditioned Galerkin projection ($m=0$). In fact, the preconditioner improves the quality of the test space, and the choice $W_r=V_r$ (yielding the standard Galerkin projection) is sufficiently accurate for this example and for the chosen norm on $V$.

\begin{figure}[ht!]
   \centering
   \subfigure[PDF of the error. Three different preconditioners $P_m(\xi)$ are used: $m=0$ (dotted lines), $m=10$ (dashed lines) and $m=30$ (continuous lines).]{\centering
     \includegraphics[width=.45\textwidth]{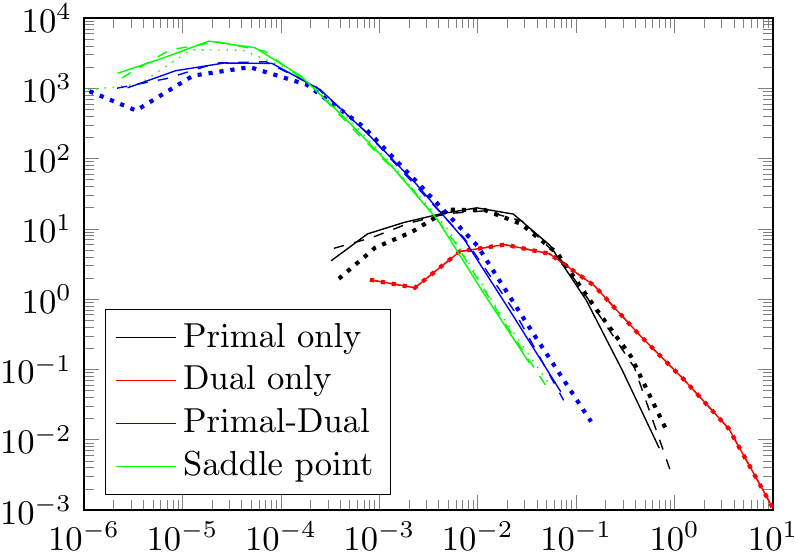}
     \label{fig:OPUS_proj_a}
   }~~
   \subfigure[ $L^\infty$ and $L^2$ norm of the error.]{
     \centering
     \includegraphics[scale=0.85]{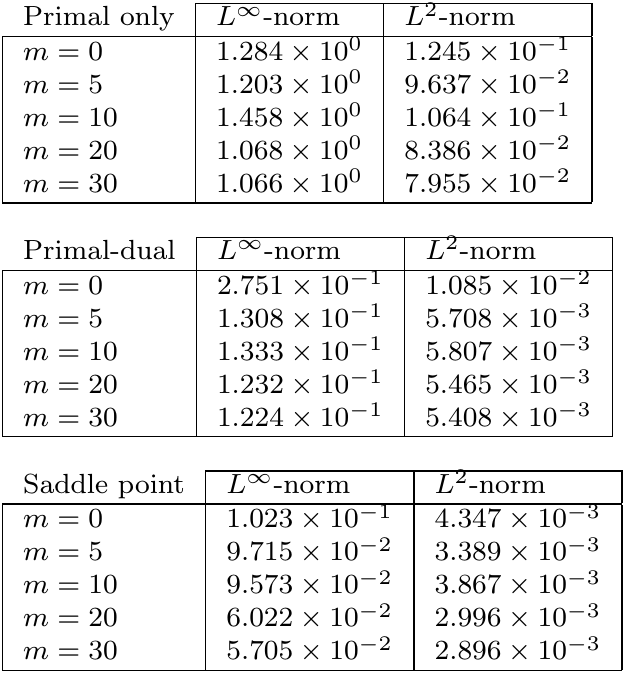}
     \label{fig:OPUS_proj_b}
   }
   \caption{Application 2: Probability density function, $L^\infty$ norm and $L^2$ norm of the error $\| s(\xi)-\widetilde s(\xi)\|_Z$ estimated from a samples set of size $10^4$.}
   \label{fig:OPUS_proj}
\end{figure}

We discuss now the quality of the error estimate $\Delta(\xi)$ for the variable of interest. Since in this application the constant $\alpha(\xi)$ can not be easily computed, we consider surrogates for \eqref{eq:standard_error_estimate} and \eqref{eq:error_estimate_General_SaddlePoint} using a preconditoner $P_m(\xi)$. We consider
\begin{equation}\label{eq:standard_error_estimate_precond}
  \Delta(\xi) = \|P_m(\xi)( A(\xi) {u_r}(\xi) - b(\xi) ) \|_{V_0} \| L(\xi)^* - A(\xi)^* {Q_k}(\xi) \|_{Z'\rightarrow V_0'}
\end{equation}
for the primal-dual method, and 
\begin{equation}\label{eq:error_estimate_General_SaddlePoint_precond}
  \Delta(\xi) = \| P_m(\xi)( A(\xi)t_{r,p}(\xi) -b(\xi) ) \|_{V_0} \sup_{0\neq z'\in Z'}\inf_{y\in T_p}\frac{\| L(\xi)^* z' -A(\xi)^*y \|_{V_0'}}{\| z' \|_{Z'}}  
 \end{equation}
for the saddle point method. Figure \ref{fig:OPUS_proj_ErrorEst} shows statistics of the effectivity index $\eta(\xi)=\Delta(\xi)/\| s(\xi) - \widetilde s(\xi) \|_Z$ for different numbers $m$ of interpolation points for the preconditioner. We see that the max-min ratio and the normalized standard deviation are decreasing with $m$: this indicates an improvement of the error estimate. Furthermore, the mean value of $\eta(\xi)$ seems to converge (with  $m$) to 19.5 for the primal-dual method, and to  13.8 for the saddle point method. In fact, with a good preconditioner, $\| P_m(\xi)( A(\xi){ u_r(\xi)} -b(\xi) ) \|_{V_0}$ (or $\| P_m(\xi)( A(\xi) t_{r,p}(\xi) -b(\xi) ) \|_{V_0}$) is expected to be a good approximation of the primal error $\| u(\xi) -{  u_r(\xi)}\|_{V_0}$ (or $\| u(\xi) - t_{r,p}(\xi)\|_{V_0}$), but this does not ensure that the effectivity index $\eta(\xi)$ will converge to 1.

\begin{figure}[ht!]
   \centering
   \centering
   \subfigure[PDF of $\eta(\xi)$ for the primal-dual methods and the saddle point methods. Three different preconditioners $P_m(\xi)$ are used: $m=0$ (dotted lines), $m=10$ (dashed lines) and $m=30$ (continuous lines)]{\centering
     \includegraphics[width=.45\textwidth]{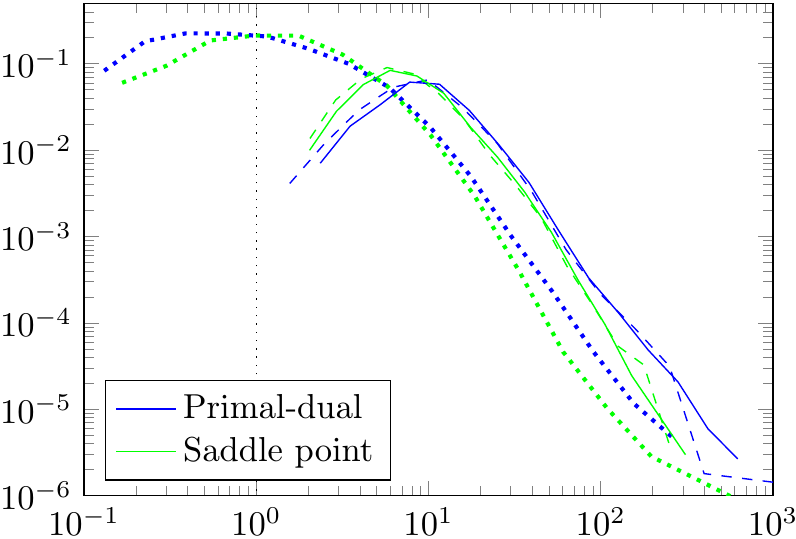}
     \label{fig:OPUS_proj_est_a}
   }~~
   \subfigure[Statistics of the effectivity index $\eta(\xi)$ for the primal-dual method and the saddle point method.]{
     \centering
     \includegraphics[scale=0.85]{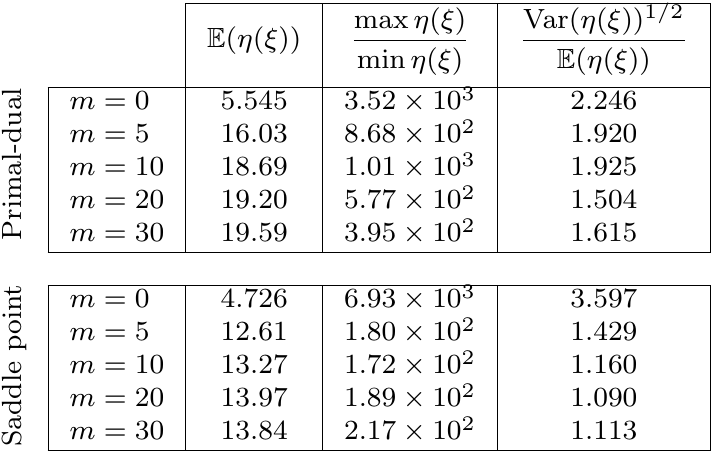}
     \label{fig:OPUS_proj_est_b}
   }
   \caption{Application 2: PDF, mean, max-min ratio and normalized standard deviation of the effectivity index $\eta(\xi)=\Delta(\xi)/\| s(\xi) - \widetilde s(\xi) \|_Z$. Here, $\Delta(\xi)$ is defined by \eqref{eq:standard_error_estimate_precond} for the primal-dual method and by \eqref{eq:error_estimate_General_SaddlePoint_precond} for the saddle point method.}
   \label{fig:OPUS_proj_ErrorEst}
\end{figure}

\subsubsection{Partial conclusions and remarks}\label{sec:Numerical_Proj_conclusion}

In both numerical examples, the {saddle point} method provides the most accurate estimation for the variable of interest. Let us note that the {saddle point} problem requires the solution of a dense linear system of size $(r+k)$ for the symmetric and coercive case, and of size $(2r+k)$ for the general case. When using Gauss elimination method for the solution of those systems, the complexity is either in $ C (r+k)^3 $ or $ C (2r+k)^3 $ (with $C=2/3$), which is larger than the complexity of the {primal-dual} method $ C (r^3+k^3)$. However, in the case where the primal and dual approximation spaces have the same dimension $r=k$, the {saddle point} method is only $4$ times (in the symmetric and coercive case) or $13.5$ times (in the general case) more expensive. 

{For the present applications, we showed that the preconditioner slightly improves the quality of the estimation $\widetilde s(\xi)$, and of the error estimate $\Delta(\xi)$. Since the construction of the preconditionner yields a significant increase in computational and memory costs (see \cite{zahm2015interpolation}), the preconditioning is not mandatory here. Nevertheless, these results revealed the important role of the test space $W_r(\xi)$ to reduce the projection error. The preconditioner used for constructing $W_r(\xi)$ can be improved, for example with a better selection of the interpolation point for $P_m(\xi)$, see Equation \eqref{eq:Wr_PrecondVr}. Note also that alternative methods can be also applied  for constructing $W_r(\xi)$, such as the subspace interpolation method proposed in \cite{amsallem2008interpolation}.}

%Furthermore, we showed that the preconditioner slightly improves the quality of the estimation $\widetilde s(\xi)$, and of the error estimate $\Delta(\xi)$. Since the construction of the preconditioner yields a significant increase in computational and memory costs (see \cite{zahm2015interpolation}), {preconditioning is not mandatory for the present applications. 
%However, this numerical analysis revealed that it is possible to reduced the projection error by improving the test space $W_r(\xi)$. 
%The method we have proposed for constructing $W_r(\xi)$ can be improved, for example with a better selection of the interpolation point for the preconditioner $P_m(\xi)$, see Equation \eqref{eq:Wr_PrecondVr}.
%Note also that the use of the preconditioner is not the only possibility for constructing $W_r(\xi)$ and alternative methods can be also applied, such as the subspace interpolation method proposed in \cite{amsallem2008interpolation}.
%}

%%%%%%%%%%%%%%%%%%%%%%%%%%%%%%%%%%%%%%%%%%%%%%%%%%%%%%%%%%%%%%%%%%%%%%%%%%%%%%%%%%%%%%%%%%
%%%%%%%%%%%%%%%%%%%%%%%%%%%%  CONSTRUCT  %%%%%%%%%%%%%%%%%%%%%%%%%%%%%%%%%%%%%%%%%%%%%%%%%%
%%%%%%%%%%%%%%%%%%%%%%%%%%%%%%%%%%%%%%%%%%%%%%%%%%%%%%%%%%%%%%%%%%%%%%%%%%%%%%%%%%%%%%%%%%%
\subsection{Greedy construction of the reduced spaces}\label{appligreedy}

We now consider the greedy construction of the reduced spaces by Algorithms \ref{alg:Simultaneous} or \ref{alg:Alternated}. For the two considered applications, we show the convergence of the error estimate with respect to the complexity of the {offline} and of the {online} phase. For the sake of simplicity, we measure the complexity of the {offline} phase with the number of operator factorizations (this corresponds to the number of iterations $I$ of Algorithms \ref{alg:Simultaneous} and \ref{alg:Alternated}). Of course exact estimation of the {offline} complexity should take into account many other steps (for example, the computation of $\Delta(\xi)$, of the preconditioner, etc), but the operator factorization is considered, for large scale applications, as the main source of computation cost. For the {online} complexity, we only consider the computation cost for the solution of one reduced system, see Section \ref{sec:Numerical_Proj_conclusion}. Here we do not take into account the complexity for assembling the reduced systems although it may be a significant part of the complexity for ``not so reduced'' systems of equations.

\subsubsection{Application 1}

Figure \ref{fig:BRIDGE_RB} shows the convergence of $\sup_\xi \Delta(\xi)$ with respect to the {offline} and {online} complexities (as defined above). In Figure \ref{fig:BRIDGE_RB_OFFLINE}, we see that the saddle point method (dashed lines) always provides lower values for the error estimate compared to the primal-dual method (continuous lines). However, as already mentioned, the saddle point method requires the solution of larger reduced systems during the \emph{online} phase. Therefore, the primal-dual method can sometimes provide lower error estimates (see the blue and red curves of Figure \ref{fig:BRIDGE_RB_ONLINE}) for the same \emph{online} complexity.

The simultaneous construction of $V_r$ and $W_k^Q$ with full dual enrichment \eqref{eq:Enrich_Wk} (green curves) yields a very fast convergence of the error estimate during the \emph{offline} phase, see Figure \ref{fig:BRIDGE_RB_OFFLINE}). But the rapid increase of $\text{dim}(W_k^Q)$ leads to high \emph{online} complexity, so that this strategy becomes non competitive during the \emph{online} phase, see Figure \ref{fig:BRIDGE_RB_ONLINE}.

We compare now the alternate and the simultaneous construction of $V_r$ and $W_k^Q$ with partial dual enrichment \eqref{eq:Enrich_Wk_one} (red and blue curves in Figure \ref{fig:BRIDGE_RB}). 
The initial idea of the alternate construction is to build reduced spaces of better quality. Indeed, since the evaluation points of the primal solution are different from the one of the dual solution, the reduced spaces are expected to contain ``complementary information'' for the approximation of the variable of interest. 
In practice, we observe in Figure \ref{fig:BRIDGE_RB_OFFLINE} that the alternate construction is (two times) more expensive during the \emph{offline} phase, but the resulting error estimate behaves very similarly to the simultaneous strategy, see Figure \ref{fig:BRIDGE_RB_ONLINE}. We conclude that the alternate strategy is not relevant for this application.

Furthermore, let us note that after iteration 50 of the greedy algorithm, the rate of convergence of the dashed red curve of Figure \ref{fig:BRIDGE_RB_OFFLINE} (i.e. the simultaneous construction with partial dual enrichment using the saddle point method) rapidly increases. A possible explanation  is that the dimension of the dual approximation space is large enough to reproduce correctly the dual variable, which requires a dimension higher than $l=44$. The same observation can be done for the alternative strategy (the dashed blue curve) after  iteration $100$ (which corresponds to $\text{dim}(W_k^Q)\geq 50$). Also, we note that the primal-dual method does not present this behavior.

\begin{figure}[ht!]
   \centering
   \subfigure[Maximum value of the error estimate $\Delta(\xi)$ with respect to the number of operator factorizations.]{\centering
     \includegraphics[width=.45\textwidth]{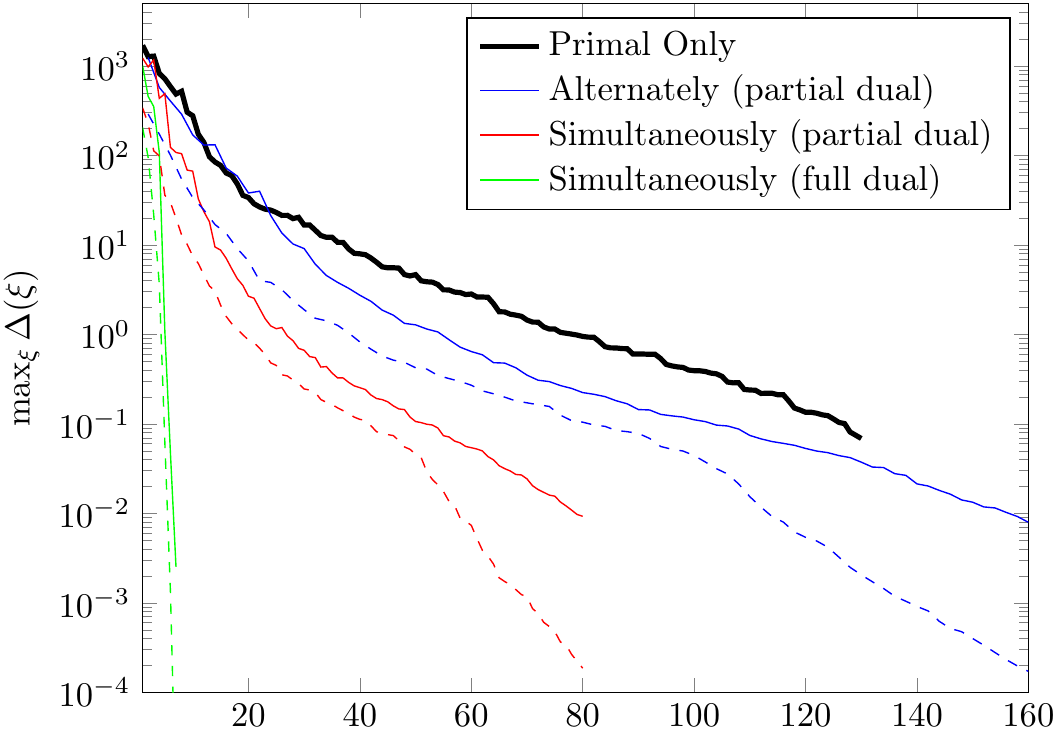}
     \label{fig:BRIDGE_RB_OFFLINE}
   }~~
   \subfigure[Maximum value of the error estimate $\Delta(\xi)$ with respect to the complexity of solving one reduced system.]{
     \centering
     \includegraphics[width=.45\textwidth]{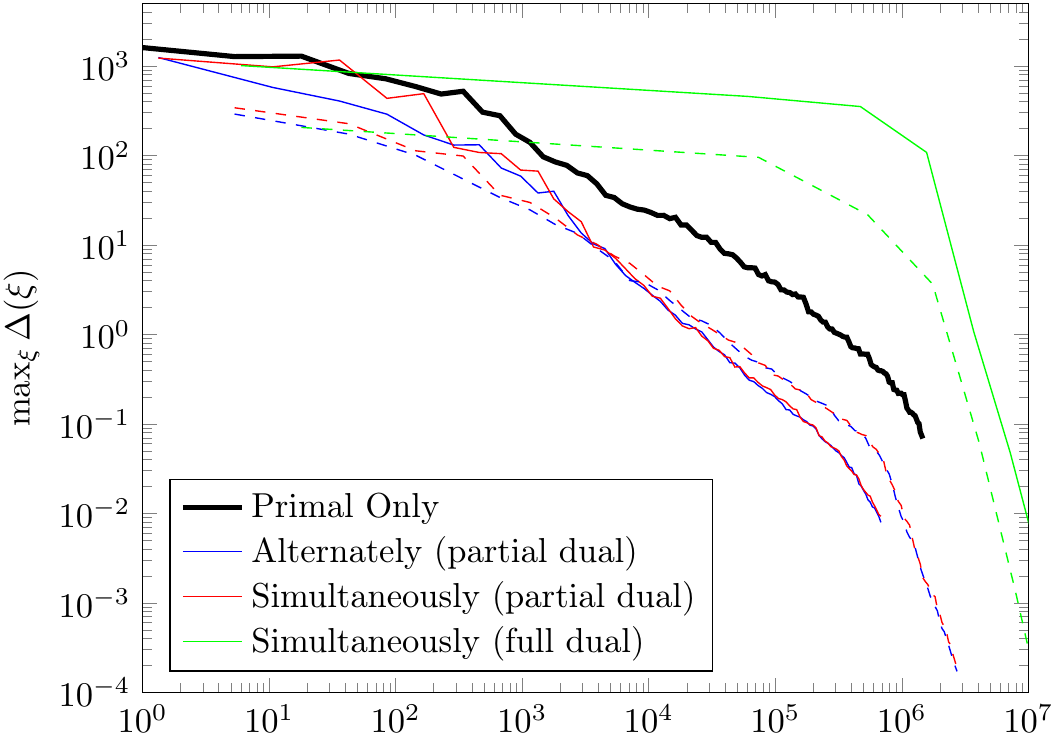}
     \label{fig:BRIDGE_RB_ONLINE}
   }
   \caption{Application 1: error estimate $\sup_{\xi}\Delta(\xi)$ with respect to the {offline} complexity (Figure \ref{fig:BRIDGE_RB_OFFLINE}) and the {online} complexity (Figure \ref{fig:BRIDGE_RB_ONLINE}). The continuous lines correspond to the {primal-dual} method, and the dashed lines correspond to the {saddle point} method. The {primal only} curves serve as reference.}
  \label{fig:BRIDGE_RB}
\end{figure}

\subsubsection{Application 2}

For the application 2, we first test Algorithms \ref{alg:Simultaneous} and \ref{alg:Alternated} with the use of  a preconditioner (defined in Section \ref{sec:appli22}).
The interpolation points for the preconditioner are the ones where the solutions (primal and dual) have been computed, see Algorithms \ref{alg:Simultaneous} and \ref{alg:Alternated}. The preconditioner is used for the definition of the test space $W_r(\xi)$, see equation \eqref{eq:Wr_PrecondVr}, and for the error estimate $\Delta(\xi)$, see equation \eqref{eq:standard_error_estimate_precond} for the primal-dual method and \eqref{eq:error_estimate_General_SaddlePoint_precond} for the saddle point method. The numerical results are given in Figure \ref{fig:OPUS_RB}. We can draw the same conclusions as for application 1.

\begin{itemize}
 \item In the {offline} phase, the saddle point method provides lower errors (Figure \ref{fig:OPUS_RB_OFFLINE}). However, the corresponding reduced systems are larger, and we see that the primal-dual method provides lower errors for the same {online} complexity, see Figure \ref{fig:OPUS_RB_ONLINE}. For this test case, the benefits (in term of accuracy) of the saddle point method does not compensate the additional {online} computational costs.
 \item The full dual enrichment yields a fast convergence during the {offline} phase, but the rapid increase of $W_k^Q$ is disadvantageous regarding the {online} complexity. However, since the dimension of the variable of interest is ``only'' $l=2$, the full dual enrichment is still an acceptable strategy (compared to the previous application).
 \item Here, the alternate strategy (blue curves) seems to yield slightly better reduced spaces compared to the simultaneous strategy, see Figure \ref{fig:OPUS_RB_ONLINE}. But this leads to higher offline costs, see Figure \ref{fig:OPUS_RB_OFFLINE}.
\end{itemize}

We also run numerical tests without using the preconditioner. In that case, we replace $P_m(\xi)$ by $R_V^{-1}$. Figure \ref{fig:OPUS_RB_withOUTPrecond} shows  numerical results which are very similar to those of Figure \ref{fig:OPUS_RB}. To illustrate the benefits of using the preconditioner, let us consider the effectivity index $\eta(\xi)=\Delta(\xi)/\| s(\xi) - \widetilde s(\xi)\|_Z$ associated to the error estimate for the variable of interest. Figure \ref{fig:RB_OPUS_ETA} shows the confidence interval $I(p)$ of probability $p$ for $\eta(\xi)$ defined as the smallest interval which satisfies
\begin{equation*}
 \mathbb{P}( \xi\in\Xi_t : \eta(\xi)\in I(p) )\geq p,
\end{equation*}
where $\mathbb{P}(A)=\#A/\#\Xi_t$ for $A\subset \Xi_t$. When using the preconditioner, we see in Figure \ref{fig:RB_OPUS_ETA} that the effectivity index is improved during the greedy iterations  in the sense that the confidence intervals are getting smaller and smaller. Also, we note that after the iteration $15$, the effectivity index is always above $1$: this indicates that the error estimate tends to be certified. Furthermore, after iteration $20$ we do not observe any further improvement, so that is seems not useful to continue enriching the preconditioner.

Let us finally note that the use of the preconditioner yields significant computational costs. Indeed, we have to store operator factorizations (in our current implementation of the method), and the computation of the interpolation of the inverse operator requires additional problems to solve (see \cite{zahm2015interpolation}). For the present application, even if the effectivity index of the error estimate is improved, the benefits of using the preconditioner remains questionable.

\begin{figure}[ht!]
   \centering
   \subfigure[Maximum value of the error estimate $\Delta(\xi)$ with respect to the number of operator factorizations.]{\centering
     \includegraphics[width=.45\textwidth]{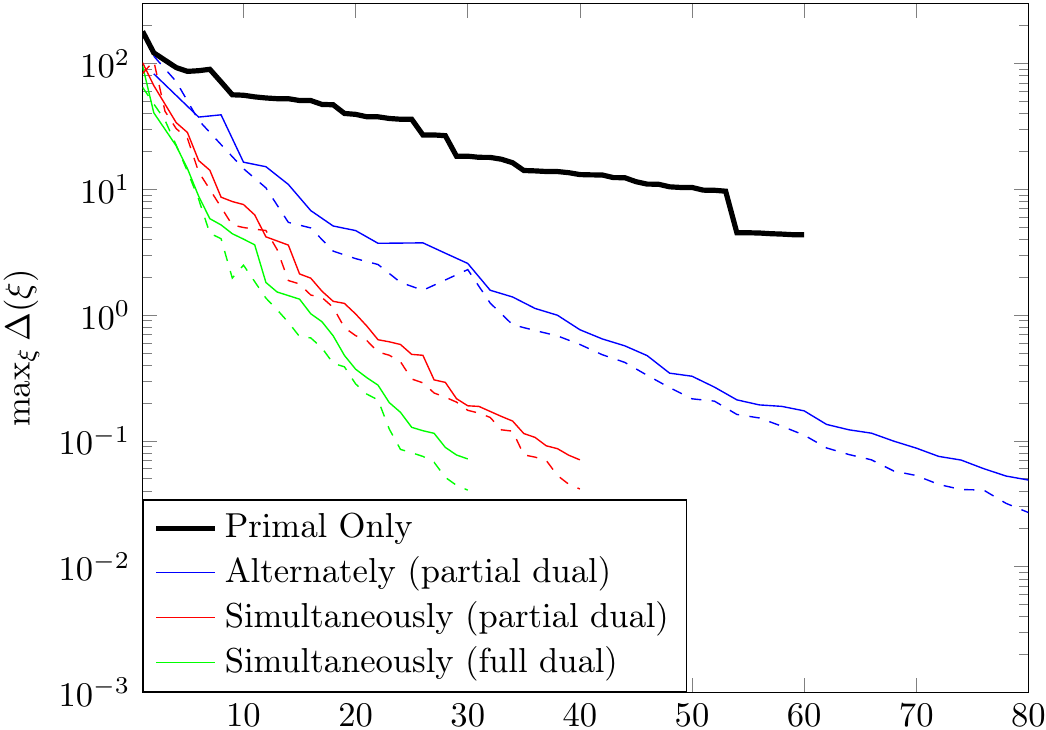}
     \label{fig:OPUS_RB_OFFLINE}
   }~~
   \subfigure[Maximum value of the error estimate $\Delta(\xi)$ with respect to the complexity for solving one reduced system.]{
     \centering
     \includegraphics[width=.45\textwidth]{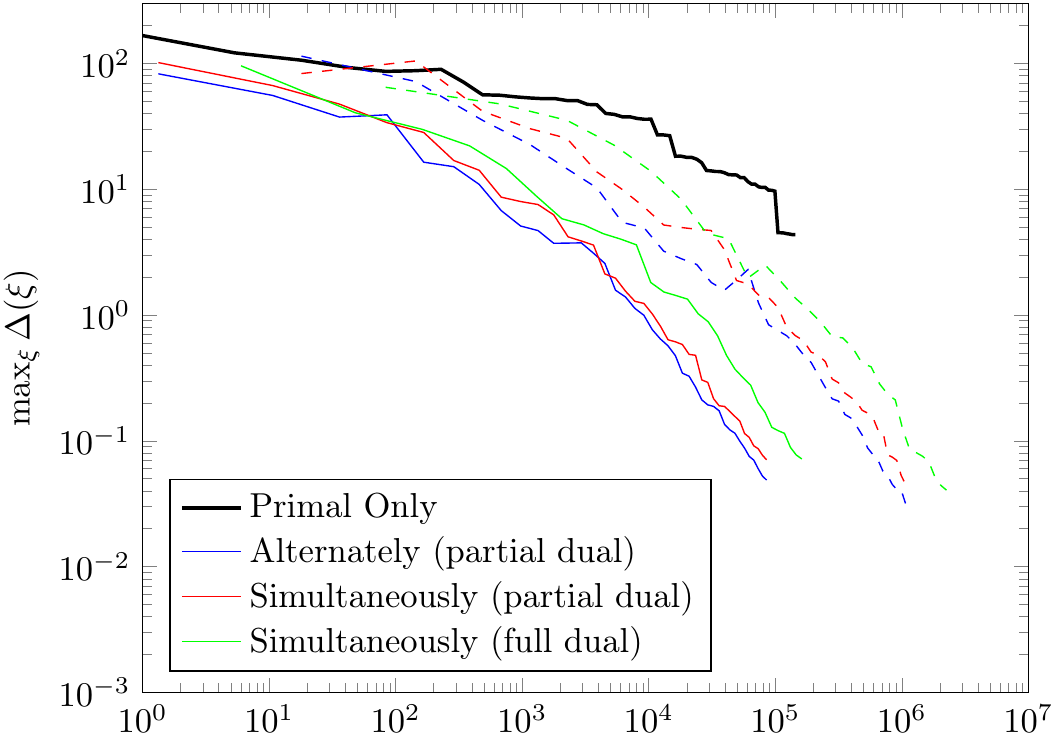}
     \label{fig:OPUS_RB_ONLINE}
   }
   \caption{Application 2 with preconditioner: error estimate $\sup_{\xi}\Delta(\xi)$ with respect to the {offline} complexity (Figure \ref{fig:OPUS_RB_OFFLINE}) and the {online} complexity (Figure \ref{fig:OPUS_RB_ONLINE}). The continuous lines correspond to the {primal-dual} method, and the dashed lines correspond to the {saddle point} method. The {primal only} curves serve as references.}
  \label{fig:OPUS_RB}
\end{figure}

\begin{figure}[ht!]
   \centering
   \subfigure[Maximum value of the error estimate $\Delta(\xi)$ with respect to the number of operator factorization.]{\centering
     \includegraphics[width=.45\textwidth]{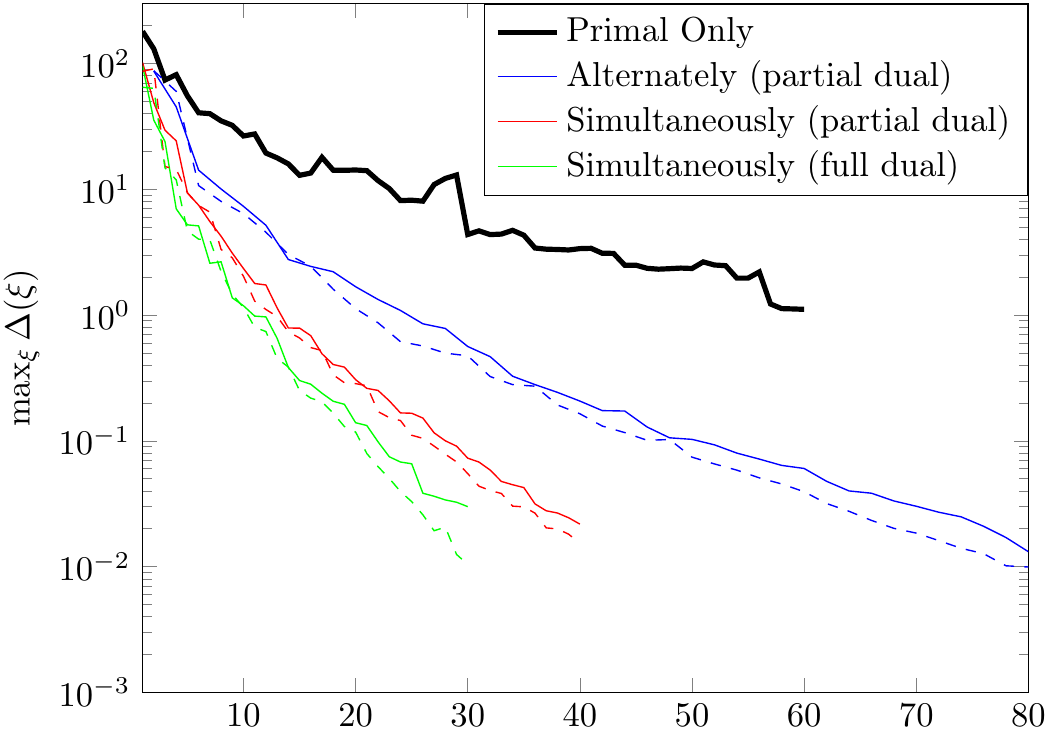}
     \label{fig:OPUS_RB_OFFLINE_withOUTPrecond}
   }~~
   \subfigure[Maximum value of the error estimate $\Delta(\xi)$ with respect to the complexity for solving one reduced system.]{
     \centering
     \includegraphics[width=.45\textwidth]{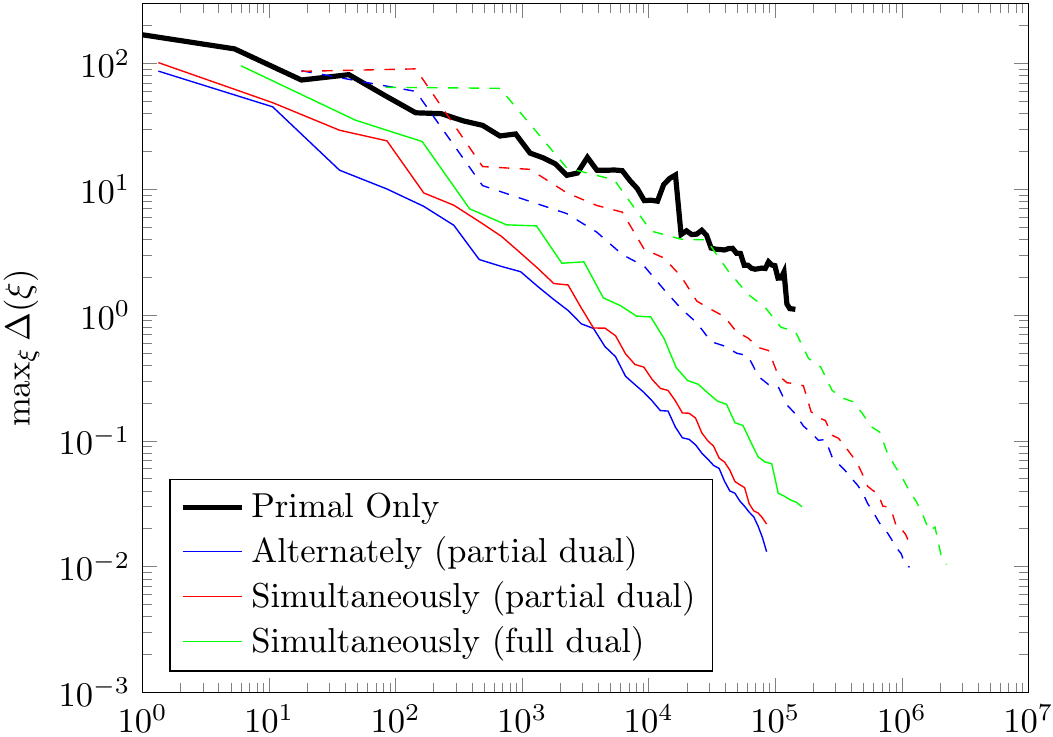}
     \label{fig:OPUS_RB_ONLINE_withOUTPrecond}
   }
   \caption{Application 2 without preconditioner: error estimate $\sup_{\xi}\Delta(\xi)$ with respect to the {offline} complexity \ref{fig:OPUS_RB_OFFLINE_withOUTPrecond} and the {online} complexity \ref{fig:OPUS_RB_ONLINE_withOUTPrecond}. The continuous lines correspond to the {primal-dual} method, and the dashed lines correspond to the {saddle point} method. The {primal only} curves serve as references.}
  \label{fig:OPUS_RB_withOUTPrecond}
\end{figure}

\begin{figure}[ht!]
   \centering
   \subfigure[Without preconditioner.]{\centering
     \includegraphics[width=.45\textwidth]{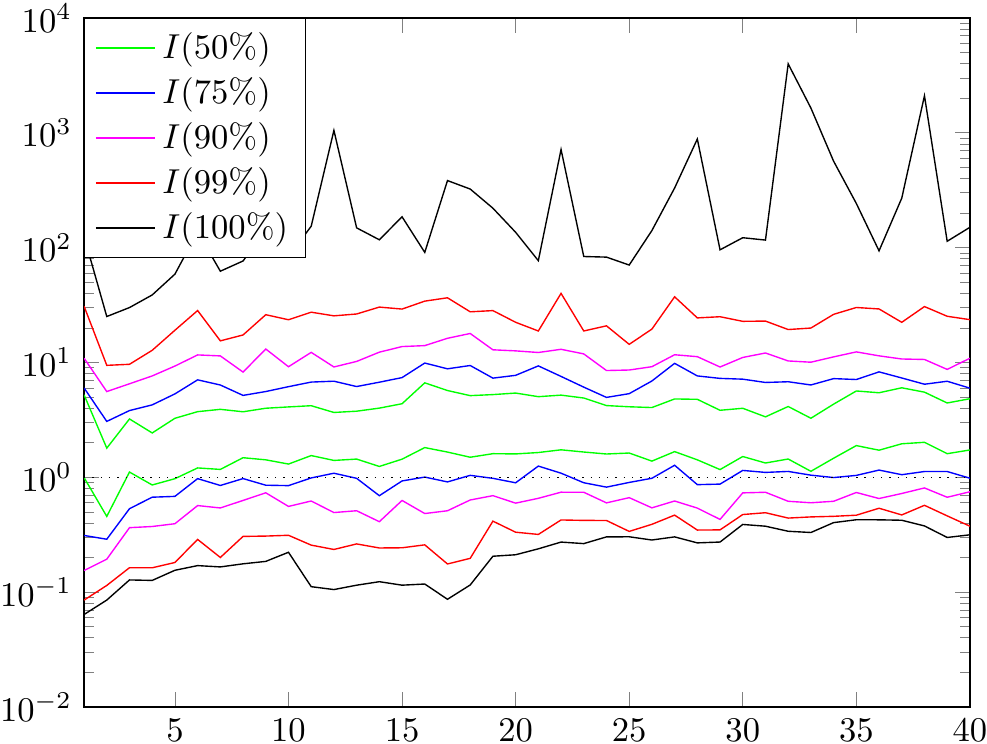}
     \label{fig:RB_OPUS_ETA_saddlePoint_WithOUTPrecond}
   }~~
   \subfigure[With preconditioner.]{
     \centering
     \includegraphics[width=.45\textwidth]{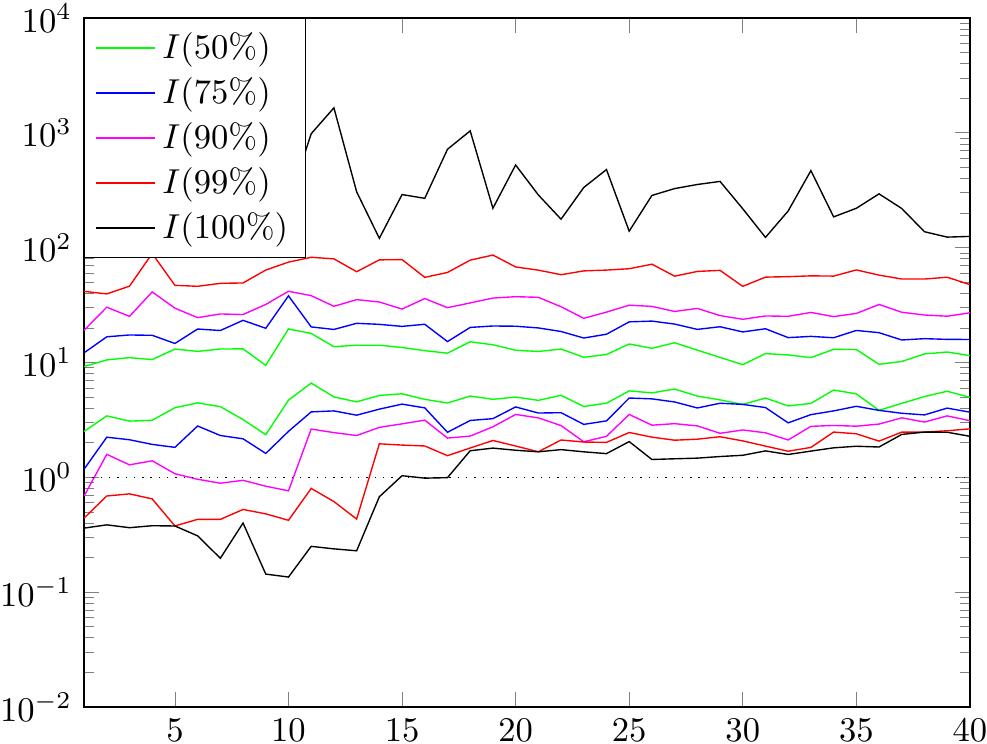}
     \label{fig:RB_OPUS_ETA_saddlePoint_WithPrecond}
   }
   \caption{Application 2: evolution with respect to the greedy iterations  of the confidence interval $I(p)$ for the effectivity index $\eta(\xi)=\Delta(\xi)/\| s(\xi) - \widetilde s(\xi)\|_Z$ for the saddle point method.}
  \label{fig:RB_OPUS_ETA}
\end{figure}

\section{Conclusion}

We have proposed and analyzed projection based methods for the estimation of vector-valued variables of interest in the context of parameter-dependent equations. This includes a generalization of the classical primal-dual method to the case of vector-valued variables of interest, and also a Petrov-Galerkin method based on a saddle point problem. Numerical results showed that the saddle point method always improves the quality of the approximation compared to the primal-dual method using the same reduced spaces. We have also derived computable error estimates and greedy algorithms for the goal-oriented construction of the reduced spaces. The performances of these approaches have been compared on numerical examples, with an analysis of both the offline complexity (construction of the reduced spaces) and the online complexity (evaluation of the reduced order model and estimation of the variable of interest for one instance of the parameter). This complexity analysis revealed that the saddle point method is preferable to the primal-dual method regarding the offline costs. However, in the situation where the reduction of the online costs matter more than the reduction of offline costs, then the primal-dual method seems to be a better option (at least for the considered applications).
{
For the considered applications, the use of preconditioners allows the construction of better reduced test spaces and also better error estimates.
Even if the additional computational costs for building the preconditioner is significant, this has demonstrated the importance of having a suitable test space and good residual based error estimates.
}

The proposed error estimates, which involve the use of Cauchy-Schwarz inequalities, are clearly not optimal. {Extending probabilistic error bounds proposed in \cite{Janon15} to the case of vector-valued variables could improve these error estimates.}

\end{document}